\documentclass[runningheads]{llncs}
\usepackage{graphicx}
\usepackage{amsmath,amsfonts,amssymb}
\usepackage{pstricks}
\usepackage[normalem]{ulem}
\def\xC{{\mathrm C}}
\def\xL{{\mathrm L}}

\def\xBV{{\mathrm BV}}
\def\xR{{\mathbb R}}
\newcommand{\bfj}{{\boldsymbol j}}
\newcommand{\bfK}{{\boldsymbol K}}
\newcommand{\bfL}{{\boldsymbol L}}
\newcommand{\bfM}{{\boldsymbol M}}
\newcommand{\bfn}{{\boldsymbol n}}
\newcommand{\bfu}{{\boldsymbol u}}
\newcommand{\bfx}{{\boldsymbol x}}
\newcommand{\dive}{{\rm div}}
\newcommand{\grad}{{\boldsymbol \nabla}}
\newcommand{\gradi}{{\boldsymbol \nabla}}
\newcommand{\norm}[1]{{\lVert #1 \rVert}}
\newcommand{\normtbv}[1]{{\lVert #1 \rVert_{\mathcal{T},t,\xBV}}}
\newcommand{\normxbv}[1]{{\lVert #1 \rVert_{\mathcal{T},x,\xBV}}}
\newcommand{\edge}{\sigma}
\newcommand{\edges}{{\mathcal E}}
\newcommand{\edgesint}{{\mathcal E}_{{\rm int}}}
\newcommand{\edgesext}{{\mathcal E}_{{\rm ext}}}
\newcommand{\mesh}{{\mathcal M}}

\newcommand{\bli}{\begin{list}{-}{\itemsep=1ex \topsep=1ex \leftmargin=0.5cm \labelwidth=0.3cm \labelsep=0.2cm \itemindent=0.cm}}
\newcommand{\dt}{\,{\rm d}t}
\newcommand{\dx}{\,{\rm d}\bfx}
\newcommand{\eg}{\emph{e.g.}}
\newcommand{\ie}{\emph{i.e.}}
\newcommand{\li}{|\hspace{-0.12em}[}
\newcommand{\ri}{]\hspace{-0.12em}|}

\begin{document}
\title{Consistent Internal Energy Based Schemes for the Compressible Euler Equations}
\titlerunning{Internal Energy Based Schemes for the Euler Equations}

\author{R. Herbin\inst{1} \and
T. Gallou\"et$^1$ \and J.-C. Latch\'e\inst{2} \and N. Therme$^2$}
\authorrunning{R. Herbin {\it et al.}}
\institute{Aix-Marseille Universit\'e, France\\
\email{thierry.gallouet@univ-amu.fr, raphaele.herbin@univ-amu.fr}\\
\and
Institut de Radioprotection et de S\^uret\'e Nucl\'eaire (IRSN), France\\
\email{jean-claude.latche@irsn.fr, ntherme@gmail.com}}
\maketitle
\begin{abstract}
Numerical schemes for the solution of the Euler equations have recently been developed, which involve the discretisation of the internal energy equation, with corrective terms to ensure the correct capture of shocks, and, more generally, the consistency in the Lax-Wendroff sense. 
These schemes may be staggered or colocated, using either structured meshes or general simplicial or tetrahedral/hexahedral meshes.
The time discretization is performed by fractional-step algorithms; these may be either based on semi-implicit pressure correction techniques or segregated in such a way that only explicit steps are involved (referred to hereafter as "explicit" variants).
In order to ensure the positivity of the density, the internal energy and the pressure, the discrete convection operators for the mass and internal energy balance equations are carefully designed; they use an upwind technique with respect to the material velocity only.
The construction of the fluxes thus does not need any Riemann or approximate Riemann solver, and yields easily implementable algorithms.
The stability is obtained without restriction on the time step for the pressure correction scheme and under a CFL-like condition for explicit variants: preservation of the integral of the total energy over the computational domain, and positivity of the density and the internal energy.
The semi-implicit first-order upwind scheme satisfies a local discrete entropy inequality.
If a MUSCL-like scheme is used in order to limit the scheme diffusion, then a weaker property holds: the entropy inequality is satisfied up to a remainder term which is shown to tend to zero with the space and time steps, if the discrete solution is controlled in $\xL^\infty$ and BV norms.
The explicit upwind variant also satisfies such a weaker property, at the price of an estimate for the velocity which could be derived from the introduction of a new stabilization term in the momentum balance.
Still for the explicit scheme, with the above-mentioned MUSCL-like scheme, the same result only holds if the ratio of the time to the space step tends to zero.

\end{abstract}
%
%
\section{Introduction} \label{sec:intro}

We address in this paper the solution of the Euler equations for an ideal gas, which read:
\begin{subequations}
\begin{align}\label{eq:pb_mass} &
\partial_t \rho + \dive( \rho\, \bfu) = 0,
\\[1ex] \label{eq:pb_mom} &
\partial_t (\rho\, \bfu) + \dive(\rho\, \bfu \otimes \bfu) + \grad p= 0,
\\[1ex] \label{eq:pb_Etot} &
\partial_t (\rho\, E) + \dive(\rho \, E \, \bfu) + \dive ( p \, \bfu )=0,
\\ \label{eq:pb_etat} &
 p=(\gamma-1)\, \rho\, e, \qquad E=\frac 1 2|\bfu|^2+e,
\end{align} \label{eq:pb}\end{subequations}
where $t$ stands for the time, $\rho$, $\bfu$, $p$, $E$ and $e$ are the density, velocity, pressure, total energy and internal energy respectively, and $\gamma > 1$ is a coefficient specific to the considered fluid.
The problem is supposed to be posed over $\Omega \times (0,T)$, where $\Omega$ is an open bounded connected subset of $\xR^d$, $1\leq d \leq 3$, and $(0,T)$ is a finite time interval.
System \eqref{eq:pb} is complemented by initial conditions for $\rho$, $e$ and $\bfu$, let us say $\rho_0$, $e_0$ and $\bfu_0$ respectively, with $\rho_0 >0$ and $e_0>0$, and by suitable boundary conditions  which we suppose to be $\bfu \cdot \bfn=0$ at any time and {\em a.e.} on $\partial\Omega$, where $\bfn$ stands for the normal vector to the boundary.

\medskip
Finite volume schemes for the solution of hyperbolic problems such as the system \eqref{eq:pb} generally use a collocated arrangement of the unknowns, which are associated to the cell centers, and apply a Godunov-like technique for the computation of the fluxes at the cells faces: the face is seen as a discontinuity line for the beginning-of-time-step numerical solution, supposed to be constant in the two adjacent cells; the value of the solution of the so-posed Riemann problem  on the discontinuity line is computed, either exactly or approximately; the numerical solution at the end-of-time-step is computed with these values, and is a piecewise constant function (see \eg\ \cite{tor-09-rie,bou-04-non} for the development of such solvers).
In one space dimension, this method consists, at least for exact Riemann solvers, in  a projection of the exact solution.
Then, thanks to the properties of the projection, this process applied to the Euler equations yields consistent schemes which preserve the non-negativity of the density and the internal energy and, for first-order variants, satisfy an entropy inequality.
The price to pay is the computational cost of the evaluation of the fluxes, and the fact that this issue is intricate enough to put almost out of reach implicit-in-time formulations, which would allow to relax CFL time step constraints.
In addition, preserving the accuracy for low Mach number flows is a difficult task (see \eg\ \cite{gui-06-rec} and references herein).

\begin{figure}[tb]
\begin{center}
\scalebox{0.85}{
\newgray{grayml}{.9}
\newgray{grayc}{.97}
\psset{unit=.9cm}
\begin{pspicture}(0,0)(17,7)
%
\rput[bl](-0.5,1){
   \pspolygon*[linecolor=grayml](3.2,3)(6,2)(6.8,3.5)(5.4,5)
   \rput[bl](4.6,3.5){{$D_\edge$}}
   \psccurve[fillstyle=solid, fillcolor=grayc, linecolor=grayc](3.2,3)(3.2,3)(6,2)(6,2)(3.7,0.8)(1,1)(1,1)
   \rput[bl](3.2,0.9){{$D_{\edge'}$}}
   \rput[bl]{10}(2.6,1.4){$\edge'$}
   \psline[linecolor=black, linewidth=1.5pt]{-}(1,1)(6,2)(5.4,5)(0.4,4)(1,1)
   \psline[linecolor=black, linewidth=0.5pt]{-}(1,1)(5.4,5)
   \psline[linecolor=black, linewidth=0.5pt]{-}(6,2)(0.4,4)
   \rput[bl](1.1,1.6){{$\bfK$}}
   \psline[linecolor=black, linewidth=1.5pt]{-}(6,2)(9,3.5)(5.4,5)
   \psline[linecolor=black, linewidth=0.5pt]{-}(6,2)(6.8,3.5)
   \psline[linecolor=black, linewidth=0.5pt]{-}(9,3.5)(6.8,3.5)
   \psline[linecolor=black, linewidth=0.5pt]{-}(5.4,5)(6.8,3.5)
   \rput[bl](6.6,2.5){{$\bfL$}}
   \psline[linecolor=black, linewidth=1.5pt]{-}(1,1)(6,2)(4,-1)(2,-1)(1,1)
   \pscurve[linecolor=black, linewidth=0.5pt]{-}(4,-1)(4.2,0.2)(3.7,0.8)
   \pscurve[linecolor=black, linewidth=0.5pt]{-}(2,-1)(2.5,0.)(3.7,0.8)
   \pscurve[linecolor=black, linewidth=0.5pt]{-}(1,1)(3.7,0.8)(6,2)
   \rput[bl](2.1,-0.9){{$\bfM$}}
   \rput[bl]{-79}(5.9,3){$\edge$}
   \psline[linecolor=blue, linewidth=1.5pt]{->}(5.7,3.5)(6.5,3.5)   \psline[linecolor=blue, linewidth=1.5pt]{->}(5.7,3.5)(5.7,4.3)
   \psline[linecolor=blue, linewidth=1.5pt]{->}(3.5,1.5)(4.3,1.5)   \psline[linecolor=blue, linewidth=1.5pt]{->}(3.5,1.5)(3.5,2.3)
   \psline[linecolor=blue, linewidth=1.5pt]{->}(2.9,4.5)(3.7,4.5)   \psline[linecolor=blue, linewidth=1.5pt]{->}(2.9,4.5)(2.9,5.3)
   \psline[linecolor=blue, linewidth=1.5pt]{->}(0.7,2.5)(1.5,2.5)   \psline[linecolor=blue, linewidth=1.5pt]{->}(0.7,2.5)(0.7,3.3)
   \psline[linecolor=blue, linewidth=1.5pt]{->}(7.5,2.75)(8.3,2.75) \psline[linecolor=blue, linewidth=1.5pt]{->}(7.5,2.75)(7.5,3.55)
   \psline[linecolor=blue, linewidth=1.5pt]{->}(7.2,4.25)(8,4.25)   \psline[linecolor=blue, linewidth=1.5pt]{->}(7.2,4.25)(7.2,5.05)
   \psline[linecolor=blue, linewidth=1.5pt]{->}(5,0.5)(5.8,0.5)     \psline[linecolor=blue, linewidth=1.5pt]{->}(5,0.5)(5,1.3)
   \psline[linecolor=blue, linewidth=1.5pt]{->}(3,-1)(3.8,-1)       \psline[linecolor=blue, linewidth=1.5pt]{->}(3,-1)(3,-0.2)
   \psline[linecolor=blue, linewidth=1.5pt]{->}(1.5,0)(2.3,0)       \psline[linecolor=blue, linewidth=1.5pt]{->}(1.5,0)(1.5,0.8)
}
%
\rput[bl](9,0){
   \pspolygon*[linecolor=grayml](1,1.5)(6,1.5)(6,4)(1,4)
   \rput[bl](1.1,3.5){{$D_\edge$}}
   \psline[linecolor=black, linewidth=2pt]{-}(1,2.5)(6,2.5)(6,5.5)(1,5.5)(1,2.5)
   \rput[bl](1.1,5.1){{$\bfK$}}
   \psline[linecolor=black, linewidth=2pt]{-}(1,0.5)(6,0.5)(6,2.5)(1,2.5)(1,0.5)
   \rput[bl](1.1,0.6){{$\bfL$}}
   \rput(3.1,2.7){$\edge$}
   \psline[linecolor=blue, linewidth=1.5pt]{->}(1,4)(1.8,4)
   \psline[linecolor=blue, linewidth=1.5pt]{->}(1,1.5)(1.8,1.5)
   \psline[linecolor=blue, linewidth=1.5pt]{->}(6,4)(6.8,4)
   \psline[linecolor=blue, linewidth=1.5pt]{->}(6,1.5)(6.8,1.5)
   \psline[linecolor=blue, linewidth=1.5pt]{->}(3.5,5.5)(3.5,6.3)
   \psline[linecolor=blue, linewidth=1.5pt]{->}(3.5,2.5)(3.5,3.3)
   \psline[linecolor=blue, linewidth=1.5pt]{->}(3.5,0.5)(3.5,1.3)
}
\end{pspicture}
}
\caption{Meshes and unknowns
-- Left: unstructured discretizations (the present sketch illustrates the possibility, implemented in our software CALIF$^3$S  \cite{califs}, of mixing simplicial and quadrangular cells); scalars variables are associated to the primal cells (here $K$, $L$ and $M$) while velocity vectors are associated to the faces (here, $\edge$ and $\edge'$) or, equivalently, to dual cells (here, $D_\edge$ and $D_{\edge'}$).
-- Right: MAC discretization; scalars variables are associated to the primal cells and each face is associated to the component of the velocity normal to the face.}
\label{fig:mesh}
\end{center}
\end{figure}
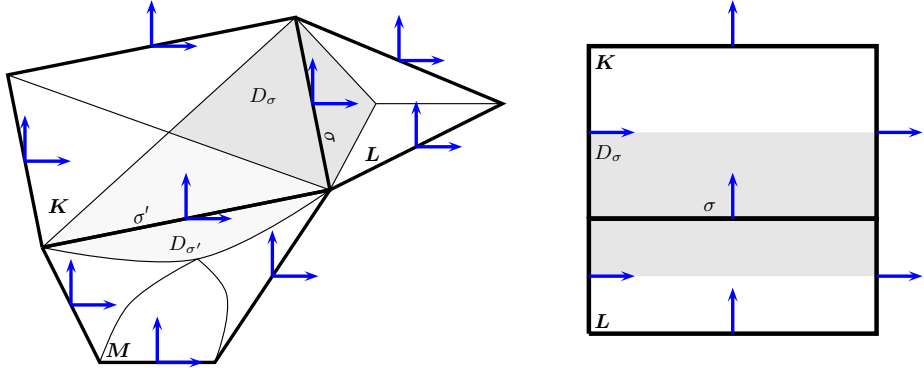

\medskip
The aim here is  first to review some recent schemes which follow a different route, and then prove some  discrete entropy estimates and/or consistency results for these schemes.
The space discretization may be colocated \cite{her-19-cel} or staggered \cite{her-14-ons,gra-16-unc,her-18-cons}:  in the colocated case, all unknowns are located at the center of the discretization cells, while in the staggered case, scalar variables are associated to cell centers while the velocity is associated to the faces, or, equivalently, to staggered mesh(es).
The use of staggered discretization for compressible flows goes back to the MAC scheme \cite{har-71-num}, and has been the subject of a wide litterature (see \cite{wes-01-pri} for a textbook and references in \cite{her-14-ons,gra-16-unc,her-18-cons}).
 Staggered discretizations have been preferred in the open source CALIF$^3$S \cite{califs} used for nuclear safety applications because the resulting semi-implicit schemes are asymptotically stable in the low Mach number regime \cite{her-19-low}. 
Two different staggered space discretizations may be considered: either the so-called Marker-And-Cell (MAC) scheme for structured grids \cite{har-65-num} or, for general meshes, a space discretization using degrees of freedom similar to low-order Rannacher-Turek \cite{ran-92-sim} or Crouzeix-Raviart \cite{cro-73-con} finite elements (see Figure \ref{fig:mesh}).
With this space discretization, the use of Riemann solvers seems difficult (scalar unknowns and velocities may still be considered as piecewise constant functions, but not associated to the same partition of the computational domain).
The positivity of the internal energy is thus ensured by a non-standard argument: the internal energy balance is discretized instead of the actual (total) energy balance \eqref{eq:pb_Etot} by a positivity-preserving scheme.
This strategy is known to lead to consistency problems (wrong shock speeds for instance), which are circumvented by some correction terms in the discrete internal energy correction.
Until now, the use of the internal energy equation associated to a consistency correction seems to be restricted to the context of Lagrangian approaches, up to a very recent work implementing a Lagrange-remap technique on staggered meshes \cite{dak-19-hig}, and some recent developements extending the techniques developped here to more genral meshes \cite{llo-18-sch}.
Two time discretizations are proposed: a pressure correction technique and a segregated scheme involving only explicit steps.
The resulting schemes offer many interesting properties: both the density and internal energy positivity are preserved, unconditionally for the pressure correction scheme and under CFL-like conditions for the segregated explicit variant, and the integral of the total energy on the computational domain is conserved (which yields a stability result); the construction of the fluxes simply relies on standard upwinding techniques of the convection operators with respect to the material velocity; finally, the space approximation, the fluxes and the choice of the internal energy balance are consistent with usual discretizations of quasi-incompressible flows, so the pressure correction scheme is asymptotic preserving  by construction in the limit of vanishing low Mach number flows (see \cite{her-19-low} for a study in the case of the barotropic Euler equations).

In addition, a discrete entropy estimate is obtained for the (upwind) pressure correction scheme, while only a conditional weak entropy estimate seems to hold for the segregated explicit variant.
Note that the schemes studied here belong to a class often referred to as "flux splitting schemes" in the literature, since they may be obtained by splitting the system by a two-step technique (usually into a "convective" and "acoustic" part), applying a standard scheme to each part (which, for the convection system, indeed yields, at first order, an upwinding with respect to the material velocity) and then summing both steps to obtain the final flux.
Works in this direction may be found in \cite{ste-81-flu,lio-93-new,zha-93-num,lio-06-seq,tor-12-flu}, and we hope that the discussion presented on the entropy may be extended in some way to these numerical methods.

The paper is organised as follows; 
Section \ref{sec:schemes} is devoted to the derivation of the previously mentioned internal energy based schemes in the semi-discrete time setting. 
Section \ref{sec:entropy} presents some new and original results concerning some entropy estimates and/or entropy consistency which hold for both the colocated and staggered schemes, first for implicit schemes, and then for explicit schemes.
%
%
\section{Derivation of the numerical schemes} \label{sec:schemes}
\subsection{A basic result on convection operators}\label{sec:convection}

Let $\rho$ and $\bfu$ be regular respectively scalar and vector-valued functions such that
\[
\partial_t \rho + \dive(\rho \bfu)=0.
\]
Let $z$ be a regular scalar function.
Then
\begin{equation}\label{eq:tr}
\begin{array}{ll}
\mathcal{C}(z) = \partial_t(\rho z) + \dive(\rho z \bfu)
&
= \rho \bigl(\partial_t z + \bfu \cdot \grad z \bigr) + z \bigl(\partial_t \rho + \dive(\rho \bfu) \bigr)
\\[1ex] &
= \rho \bigl(\partial_t z + \bfu \cdot \grad z \bigr).
\end{array}
\end{equation}
Let $\varphi$ be a regular real function.
Then:
\[
\varphi'(z)\ \mathcal C(z)= \varphi'(z)\, \rho\, \bigl(\partial_t z + \bfu \cdot \grad z \bigr)
= \rho\, \bigl(\partial_t \varphi(z) + \bfu \cdot \grad \varphi(z) \bigr).
\]
Now, reversing the computation performed in Relation \eqref{eq:tr} with $\varphi(z)$ instead of $z$ leads to
\begin{equation}\label{eq:cvz}
\varphi'(z)\ \mathcal C(z)= \partial_t \bigl(\rho \varphi(z) \bigr) + \dive \bigl(\rho \varphi(z)\, \bfu \bigr).
\end{equation}
The following lemma states a time semi-discrete version of this computation.

\begin{lemma}[Convection operator]\label{lmm:tr}
	Let $\rho^n$, $\rho^{n+1}$, $z^n$ and $z^{n+1}$ be regular scalar functions, let $\bfu$ be a regular vector-valued function and let $\varphi$ be a twice-differentiable real function.
	Let us suppose that
	\begin{equation}
		\label{eq:lm_mass}
		\frac 1 {\delta t}\ (\rho^{n+1}-\rho^n) + \dive(\rho^{n+1}\bfu)=0,
	\end{equation}
	with $\delta t$ a positive real number.
	Then
	\begin{multline} \label{ident-convec}
		\varphi'(z^{n+1})\ \bigl[\frac 1 {\delta t}\ (\rho^{n+1} z^{n+1}-\rho^n z^n) + \dive(\rho^{n+1} z^{n+1} \bfu)\bigr]
		\\
		=\frac 1 {\delta t}\ \Bigl(\rho^{n+1} \varphi(z^{n+1})-\rho^n \varphi(z^n)\Bigr) + \dive \bigl(\rho^{n+1} \varphi(z^{n+1})\, \bfu \bigr) + \mathcal R^n,
	\end{multline}
	with
	\[
		\mathcal R^n = \frac 1 {2\, \delta t} \rho^n \varphi''(\bar z)\,(z^{n+1}-z^n)^2,\quad \bar z = \theta z^n +(1-\theta)z^{n+1},\quad \theta \in [0,1].
	\]
\end{lemma}

\begin{proof}
We first begin by deriving a discrete analogue to Identity \eqref{eq:tr}:
\begin{equation}
\begin{array}{l}\displaystyle
\frac 1 {\delta t}\ (\rho^{n+1} z^{n+1}-\rho^n z^n) + \dive(\rho^{n+1} z^{n+1} \bfu)
\\[2ex] \displaystyle
= \frac 1 {\delta t}\ \rho^n\ (z^{n+1}- z^n) + \rho^{n+1} \bfu \cdot \grad z^{n+1}
+ z^{n+1} \Bigl[ \frac 1 {\delta t}\ (\rho^{n+1}-\rho^n) + \dive(\rho^{n+1}\bfu) \Bigr]
\\[2ex] \displaystyle
= \frac 1 {\delta t}\ \rho^n\ (z^{n+1}- z^n) + \rho^{n+1} \bfu \cdot \grad z^{n+1}.
\end{array}
\end{equation}
Then the result follows by multiplying this relation by $\varphi'(z^{n+1})$, using a Taylor expansion for the first term and the same combination of partial derivative as in the continuous case for the second term, and finally, still as in the continuous cas, by performing this computation in the reverse sense with $\varphi(z^n)$ and $\varphi(z^{n+1})$ instead of $z^n$ and $z^{n+1}$.
\end{proof}
%
%
\subsection{Internal energy formulation }

We begin with a formal reformulation of the energy equation.
Let us suppose that the solution is regular, and let $E_k$ be the kinetic energy, defined by $E_k= \frac 1 2 \, |\bfu|^2$.
Taking the inner product of \eqref{eq:pb_mom} by $\bfu$ yields, after the formal compositions of partial derivatives described in the previous section:
\begin{equation} \label{eq:pb_Ec}
\partial_t (\rho E_k) + \dive \bigl(\rho\, E_k\, \bfu \bigr) + \grad p \cdot \bfu = 0.
\end{equation}
This relation is referred to as the kinetic energy balance.
Subtracting this relation to the total energy balance \eqref{eq:pb_Etot}, we obtain the so-called internal energy balance equation:
\begin{equation}\label{eu:eq:pb_Eint}
\partial_t (\rho e) + \dive(\rho e \bfu)+ p\, \dive \bfu =0.
\end{equation}
Since,
\bli
\item as seen in the previous section, thanks to the mass balance equation, the first two terms in the left-hand side of \eqref{eu:eq:pb_Eint} may be recast as a transport operator,
\item and, from the equation of state, the pressure vanishes when $e=0$,
\end{list}
this equation implies that, if $e \geq 0$ at $t=0$ and with suitable boundary conditions, then $e$ remains non-negative at all time.
The same result would hold if \eqref{eu:eq:pb_Eint} featured a non-negative right-hand side, as for the compressible Navier-Stokes equations.
Solving  the internal energy balance \eqref{eu:eq:pb_Eint} instead of the total energy balance\eqref{eq:pb_Etot} is thus appealing, to preserve this positivity property by construction of the scheme.
In addition, it avoids introducing a space discretization for the total energy which, for a staggered discretization, combines cell-centered (internal energy and density) and face-centered (velocity) variables.
However, a raw discretization of a non-conservative equation derived from a conservative system (formally, \ie\ supposing unrealistic regularity properties of the solution) may be non-consistent (and the numerical test presented in Section \ref{sec:num} shows that, for the problem at hand, such a scheme would be unable to capture shock solutions).
To deal with this problem, we implement the following strategy:
\bli
\item First, we derive a discrete kinetic energy balance, by mimicking at the discrete level the computation leading to Equation \eqref{eq:pb_Ec}, so as to identify the terms which are likely to lead to non-consistency: the numerical diffusion in the momentum balance equation yields dissipation terms in the kinetic energy balance which are observed to behave, when the space and time step tend to zero, as measure born by the shocks which modify the jump conditions.
\item These terms are thus compensated in the internal energy balance.
\end{list}
At the fully discrete level, for staggered discretizations, the kinetic and internal energy balances are not posed on the same mesh (the dual and primal mesh respectively); however, it is possible to derive from the kinetic energy balance on the dual mesh a counterpart posed on the primal mesh and, adding to the internal energy balance yields a conservative total energy balance.
The scheme then can be proven to be consistent in the Lax-Wendroff sense to the weak form of the total energy balance: for a given sequence of discrete solutions (obtained with a sequence of discretizations whith  space and time steps tending to zero) controlled and converging to a limit in suitable norms (namely, uniformly bounded and converging in $L^r$ norms, for $r\in [1,+\infty)$), we show that the limit is a weak solution to the Euler equations \cite{her-18-cons,her-19-cons}.
In the colocated case, both kinetic and internal energy balances are posed on the same mesh and a discrete local total energy balance is easily recovered \cite{her-19-cel}.
%
%
\subsection{The time semi-discrete pressure correction scheme}

This semi-discrete pressure correction scheme takes the following general form:
\begin{subequations}
\begin{align} \label{eq:sch_pred} &
\frac 1 {\delta t} (\rho^n\, \tilde \bfu^{n+1}-\rho^{n-1}\,\bfu^n) + \dive(\rho^n\, \bfu^n \otimes \tilde \bfu^{n+1}) + \zeta^n \grad p^n= 0,
\displaybreak[1] \\[2ex] \label{eq:sch_corr} &
\frac 1 {\delta t} \rho^n\, (\bfu^{n+1}-\tilde \bfu^{n+1}) + \grad p^{n+1} - \zeta^n \grad p^n= 0,
\displaybreak[1] \\[1ex] \label{eq:sch_mass} &
\frac 1 {\delta t} (\rho^{n+1}-\rho^n) + \dive( \rho^{n+1}\, \bfu^{n+1}) = 0,
\displaybreak[1] \\[1ex] \label{eq:sch_eint} &
\frac 1 {\delta t} (\rho^{n+1}\, e^{n+1}-\rho^n\, e^n) + \dive(\rho^{n+1}\, e^{n+1}\, \bfu^{n+1}) + p^{n+1} \dive \bfu^{n+1} = S^{n+1},
\displaybreak[1] \\[1ex] \label{eq:sch_etat} &
p^{n+1}=(\gamma-1)\, \rho^{n+1}\, e^{n+1}.
\end{align} \label{eq:sch}\end{subequations}
Solving the first equation yields a tentative velocity $\tilde \bfu^{n+1}$; this is the velocity prediction step, which is decoupled from the other equations of the system.
Equations \eqref{eq:sch_corr}-\eqref{eq:sch_etat} constitute the correction step and are solved simultaneously; in the relation \eqref{eq:sch_eint}, the term $\rho^{n+1}\, e^{n+1}$ is recast as a function of the pressure only thanks to the equation of state \eqref{eq:pb_etat} and the velocity $\bfu^{n+1}$ is eliminated thanks to the divergence of \eqref{eq:sch_corr} divided by $\rho^n$. 
The result is a nonlinear and nonconservative elliptic problem for the pressure only.
This process must be performed at the fully discrete level to preserve the properties of the scheme.
The coefficient $\zeta^n$ in Equation \eqref{eq:sch_pred} and the correction term $S^{n+1}$ in \eqref{eq:sch_eint} are chosen so as to ensure stability and consistency, as shown below.
The first step of this process is to obtain a discrete kinetic energy balance.
To this purpose, let us multiply \eqref{eq:sch_pred} by $\tilde \bfu^{n+1}$ and apply Lemma \ref{lmm:tr} component by component, with $\varphi(s)=\frac 1 2 s^2$.
We get:
\begin{equation}\label{eq:ec1}
\frac 1 {2\,\delta t} \, \bigl(\rho^n\, |\tilde \bfu^{n+1}|^2-\rho^{n-1}\,|\bfu^n|^2 \bigr) 
+ \frac 1 2 \dive\bigl(\rho^n\, |\tilde \bfu^{n+1}|^2 \bfu^n \bigr) + \zeta^n \grad p^n \cdot \tilde \bfu^{n+1}  + R^n_1 = 0,
\end{equation}
with
\[
R^n_1=\frac 1 {2\,\delta t} |\tilde \bfu^{n+1}-\bfu^n|^2.
\]
Note that the mass balance equation \eqref{eq:sch_mass}, which is a fundamental assumption in Lemma \ref{lmm:tr}, only holds at this stage of the algorithm with the previous time step values, hence the shift of the time level of the density in \eqref{eq:sch_pred}.
Let us now recast Equation \eqref{eq:sch_corr} as
\[
\alpha^n \bfu^{n+1} + \frac 1 {\alpha^n} \grad p^{n+1}= \alpha^n \tilde\bfu^{n+1} + \frac{\zeta^n}{\alpha^n}\grad p^n,
\quad \alpha^n=\bigl[\frac{\rho^n}{\delta t}\bigr]^{1/2}
\]
and square this relation, to get
\begin{equation}\label{eq:ec2}
\frac 1 {2\,\delta t} \, \rho^n\, |\bfu^{n+1}|^2 + \grad p^{n+1} \cdot \bfu^{n+1} + R^n_2
=\frac 1 {2\,\delta t} \, \rho^n\, |\tilde \bfu^{n+1}|^2 + \zeta^n  \grad p^n \cdot \tilde \bfu^{n+1},
\end{equation}
with
\[
R^n_2= \frac{\delta t}{\,\rho^n} |\grad p^{n+1}|^2 - (\zeta^n)^2 \frac{\delta t}{\,\rho^n} |\grad p^n|^2.
\]
Summing \eqref{eq:ec1} and \eqref{eq:ec2} yields the kinetic energy balance that we are seeking:
\[
\frac 1 {2\,\delta t} \, \bigl(\rho^n\, |\bfu^{n+1}|^2-\rho^{n-1}\,|\bfu^n|^2 \bigr) 
+ \frac 1 2 \dive\bigl(\rho^n\, |\tilde \bfu^{n+1}|^2 \bfu^n \bigr) + \grad p^{n+1} \cdot \bfu^{n+1} + R^n_1 + R^n_2 = 0.
\]
The coefficient $\zeta^n$ is then chosen in such a way that the remainder term $R^n_2$ is a difference of two consecutive time levels of the same quantity; this is the case for
\[
\zeta^n=\bigl[ \frac{\rho^n}{\rho^{n-1}} \bigr]^{1/2}.
\]
Supposing the control in $L^1(0,T,BV)$ of the pressure and in $L^\infty$ of the pressure and of the inverse of the density, the term $R^n_2$ may thus be seen to tend with zero with the discretization parameters in a distributional sense.
The term $R^n_1$  is compensated  in the internal energy balance, by choosing $S^{n+1}=R^n_1$, thus ensuring that $S^{n+1}\geq 0$.
The definition of the time-discrete scheme is now complete.
%
%
\subsection{The fully discrete pressure correction scheme}

The fully discrete scheme is obtained from System \eqref{eq:sch} by applying the following guidelines:
\bli
\item The mass and internal energy balances (\ie\ Equations \eqref{eq:sch_mass} and \eqref{eq:sch_eint} respectively) are discretized on the primal mesh, while the velocity prediction \eqref{eq:sch_pred} and correction \eqref{eq:sch_corr} are discretized on the dual mesh(es).
The equation of state only involves cell quantities, and its expression is obtained by writing \eqref{eq:sch_etat} for these latter.

\item The space arrangement of the unknowns (density discretized at the cell and velocity at the faces) yields a natural expression of the mass fluxes in the mass balance, performed by a first-order upwind scheme (with respect to the velocity).
By construction, the density is thus non-negative; in fact at the discrete level, it remains positive if the initial density is positive.
The discrete mass balance equation on the cell $K$ whose measure is denoted by $|K|$ takes the form:
\begin{equation}\label{eq:mass_d}
\frac{|K|}{\delta t}\, (\rho_K^{n+1} - \rho_K^n) + \sum_{\edge \in \edges(K)} F^{n+1}_{K,\edge} =0,
\end{equation}
where $\edges(K)$ denotes the set of edges of $K$ and $F_{K,\edge}$ is the mass flux across $\edge$ outward $K$.

\item Let $\mathcal C_K(e^{n+1})$ denote the sum of the discrete  time-derivative and convection operator in the internal energy balance \eqref{eq:sch_eint}; this quantity reads:
\[
\mathcal C_K(e^{n+1}) = \frac{|K|}{\delta t}\, (\rho_K^{n+1} e^{n+1}_K - \rho_K^n e^n_K) + \sum_{\edge \in \edges(K)} F_{K,\edge} e^{n+1}_\edge,
\]
where $e^{n+1}_\edge$ is the upwind approximation of $e^{n+1}$ at $\edge$ with respect to $F^{n+1}_{K,\edge}$ (or, equivalently, since the density is positive, with respect to the velocity).
The structure of $\mathcal C_K(e^{n+1})$ (precisely speaking, the fact that $\mathcal C_K(e^{n+1})$ vanishes thanks to the mass balance if the internal energy $e^{n+1}$ is constant over $\Omega$) was shown in \cite{lar-91-how} to yield a positivity-preserving operator, and is also a necessary condition for a fully discrete version of Lemma \ref{lmm:tr} to hold; this is of course linked since both results rely on the possibility to recast $\mathcal C_K$ as a transport operator, and the positivity-preserving property of $\mathcal C_K$ may be proved by applying Lemma \ref{lmm:tr} with $\varphi(s)=\min(s,0)^2$.
Once again, thanks to the arrangement of the unknowns, a natural discretization for $\dive \bfu^{n+1}$ is available.
Since $p^{n+1}$ is a function of $e^{n+1}$ (given by the equation of state) which vanishes for $e^{n+1}=0$ and since the corrective term is non-negative, we are able to show that the discrete internal energy is kept positive by the scheme.

\item To allow to derive a discrete kinetic energy balance, the same structure is needed for the time-derivative and convection operator in the velocity prediction step \eqref{eq:sch_pred}.
This raises a difficulty since this equation is posed on the dual mesh, and thus we need an analogue of the mass balance \eqref{eq:mass_d} to also hold on this mesh.
The way to build the face density and the mass fluxes across the faces of the dual mesh for such a relation to hold, while still ensuring the scheme consistency, is a central ingredient of the scheme; it is detailed in \cite{gal-10-kin} for the MAC discretization and in \cite{lat-18-conv} for unstructured discretizations.

Once the face density is defined, the discretization of the coefficient $\zeta^n$ is straightforward.
In order to combine the discrete equivalents of $\bfu \cdot \grad p$ (kinetic energy balance) and $p\, \dive \bfu$ (internal energy balance), the discrete gradient is defined as the transposed of the divergence operator with respect to the $L^2$ inner product (if $\bfu \cdot \grad p + p\, \dive \bfu=\dive(p\,\bfu)$, the integral of this quantity over the computational domain vanishes when the normal velocity is prescribed to zero at the boundary).
Note that this definition is consistent with the usual treatment in the incompressible case, and is a key ingredient for the scheme to be asymptotic preserving in the limit of vanishing Mach number flows \cite{her-19-low}.
As in the incompressible case, it also allows to control the $L^2$ norm of the pressure by a weak norm of its gradient, which is central for convergence studies; with this respect, a discrete {\em inf-sup} condition is required in some sense, which is true for staggered discretizations.
\end{list}
%
%
\subsection{A segregated variant}

A variant of the proposed scheme which consists only in explicit steps (in the sense that these steps do not require the solution of any linear or non-linear algebraic system) reads, in the time semi-discrete setting:
\begin{subequations}
\begin{align}
\label{eq:sche_mass} &
\frac 1 {\delta t} (\rho^{n+1}-\rho^n) + \dive( \rho^n\, \bfu^n) = 0,
\\[1ex] \label{eq:sche_eint} &
\frac 1 {\delta t} (\rho^{n+1}\, e^{n+1}-\rho^n\, e^n) + \dive(\rho^n\, e^n\, \bfu^n) + p^n \dive \bfu^n= S^n,
\\[1ex] \label{eq:sche_etat} &
p^{n+1}=(\gamma-1)\, \rho^{n+1}\, e^{n+1},
\\[1ex]  \label{eq:sche_mom} &
\frac 1 {\delta t} (\rho^{n+1}\, \bfu^{n+1}-\rho^n\,\bfu^n) + \dive(\rho^n\, \bfu^n \otimes \bfu^n) + \grad p^{n+1}= 0.
\end{align} \label{eq:sche}\end{subequations}
The update of the pressure before the solution of the momentum balance equation is crucial in our derivation of entropy estimates (see Section \ref{sec:entropy} below).
This issue seems to be supported by numerical experiments: omitting it, we observe the appearance of non-entropic discontinuities in rarefaction waves \cite{her-18-cons}.

\medskip
The space discretization differs from the pressure correction scheme described in the above section in two points:
\bli
\item the discretization of the convection operator in the momentum balance equation \eqref{eq:sche_mom} is performed by the first order upwind scheme (still with respect to the material velocity $\bfu^n$),
\item the corrective term $S^n$ is still obtained by deriving a kinetic energy balance multiplying Equation \eqref{eq:sche_mom} by $\bfu^{n+1}$, but its expression is quite different, due to the time-level used in the convection operator.
The time-discretization is now anti-diffusive but, as usual for explicit schemes, this anti-diffusion is counterbalanced by the diffusion in the approximation of the convection (hence the upwinding) and $S^n$ is non-negative only under a CFL condition.
\end{list}
%
%
\subsection{A numerical test}\label{sec:num}

In this section, we reproduce a test performed in \cite{her-14-ons} to assess the behaviour of the scheme on a one dimensional Riemann problem.
We choose initial conditions such that the structure of the solution consists in two shock waves, separated by the contact discontinuity, with sufficiently strong shocks to allow an easy discrimination of correct numerical solutions.
These initial conditions are those proposed in \cite[chapter 4]{tor-09-rie}, for the test referred to as Test 5.
The computations are performed with the open-source software CALIF$^3$S \cite{califs}.

\medskip
The density fields obtained with $h=1/2000$ (or a number of cells $n=2000$) at $t=0.035$, with and without assembling the corrective source term in the internal energy balance, together with the analytical solution, are shown on Figure \ref{cvns:fig:ww}.
We observe that both schemes seem to converge, but the corrective term is necessary to obtain the right solution.
Without a corrective term, one can check that the obtained solution is not a weak solution to the Euler system (Rankine-Hugoniot conditions are not verified).
We also observe that the scheme is rather diffusive especially at contact discontinuities for which the beneficial compressive effect of the shocks does not apply; this may be cured in the explicit variant by implementing MUSCL-like algorithms \cite{gas-18-mus}.

\begin{figure}[htb]
\begin{center}
\scalebox{0.7}{\includegraphics*[2.2cm,1.5cm][14.5cm,10.5cm]{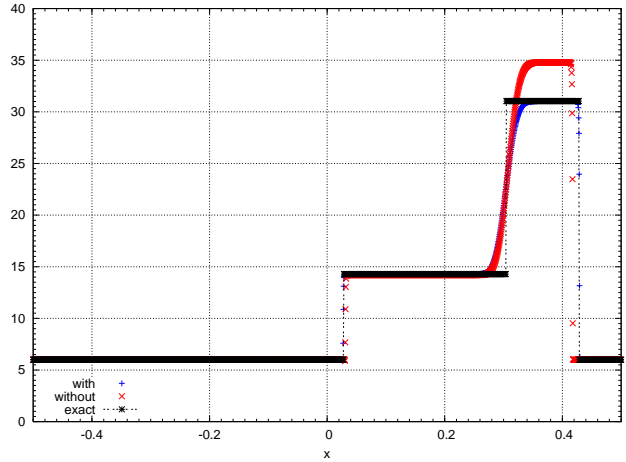}}
\end{center}
\caption{\label{cvns:fig:ww}
Test 5 of \cite[chapter 4]{tor-09-rie} - Density obtained with $n=2000$ cells, with and without corrective source terms in pressure correction scheme, and analytical solution.
}
\end{figure}

Extensive multidimensional tests were performed in both the staggered case \cite{gra-16-unc} and the colocated case \cite{her-19-cel}. 
%
%
\section{Entropy}\label{sec:entropy}

In the case of regular solutions to the Euler equations \eqref{eq:pb}, an additional conservation law can be written for an additional quantity called entropy; however, in the presence of shock waves, the (mathematical) entropy decreases. 
It is now known that weak solutions of the Euler system satisfying an entropy inequality may be non unique \cite{chi-15-wea}; nevertheless,  entropy inequalities play an important role in providing global stability estimates \cite{bou-04-non}. 

\medskip
When solving the Euler equations numerically, it is thus natural to design numerical schemes such that some entropy inequalities are satisfied by the approximate solutions; these inequalities should enable to prove that, as the mesh and time steps tend to 0, the limit of the approximate solutions, if it exists, satisfies an entropy inequality. 
A classical way of doing so is to design so-called ``entropy stable schemes" \cite{tad-16-ent}. 
Discrete entropy inequalities are known for the one dimensional case for the Godunov scheme \cite{god-59-dif} and have been derived for Roe-type schemes in the one space dimension case \cite{ism-09-aff}.
Entropy stability has also been proven in the multi-dimensional case for semi-discrete schemes on unstructured meshes \cite{mar-12-ent,ray-16-ent}.
In the sequel we show that an implicit upwind scheme (at least with the upwinding with respect to the material velocity used here) is indeed entropy stable. 
However it is not always possible to obtain entropy stability, especially for fully discrete schemes such as the explicit schemes studied below; in this case, weaker discrete entropy inequalities or estimates are obtained which allow to fulfil our goal, namely to show that the possible limits of the approximate solutions satisfy an entropy inequality. 
Such a technique was used for the convergence study of a time implicit mixed finite volume--finite element scheme for the Euler-Fourier equations \cite{fei-16-con}, with a special equation of state  which allows to obtain a priori estimates.

\medskip
Both the pressure correction scheme \eqref{eq:sch} and the segregated scheme \eqref{eq:sche} involve a discrete equivalent of the following subsystem:
\begin{subequations}\label{eq:cont}
\begin{align}\label{eq:mass} &
\partial_t \rho + \dive( \rho\, \bfu) = 0,
\\[1ex] \label{eq:e_int} &
\partial_t (\rho\, e) + \dive(\rho \, e \, \bfu) + p\, \dive (\bfu) = \mathcal R \geq 0,
\\ \label{eq:etat} &
 p=(\gamma-1)\, \rho\, e,
\end{align}\end{subequations}
with the same initial and boundary conditions as for the full system \eqref{eq:pb}.

\medskip
The derivation of an entropy for the continuous Euler system may be deduced from the subsystem \eqref{eq:cont} in the following way.
We seek an entropy function $\eta$ satisfying:
\begin{equation}\label{eq:entropy}
\partial_t \eta(\rho,e) + \dive\bigl[ \eta(\rho,e)\, \bfu \bigr] \leq 0.
\end{equation}
To this end, we introduce the functions $\varphi_\rho$ and $\varphi_e$ defined as follows:
\begin{equation}\label{eq:vphis}
\varphi_\rho(z)=z \ln(z),\quad \varphi_e(z)=\frac{-1}{\gamma -1} \ln(z), \quad \mbox{for } z >0.
\end{equation}
For regular functions, the function $\eta$ defined by
\begin{equation}\label{eq:def_eta}
\eta(\rho,e)=\varphi_\rho(\rho)+\rho \varphi_e(e)
\end{equation}
satisfies \eqref{eq:entropy}.
Indeed, multiplying \eqref{eq:mass} by $\varphi'_\rho(\rho)$, a formal computation yields:
\begin{equation}\label{eq:ent_m}
\partial_t \bigl[\varphi_\rho(\rho)\bigr] + \dive\bigl[ \varphi_\rho(\rho)\, \bfu \bigr]
+ \bigl[\rho\varphi'_\rho(\rho)-\varphi_\rho(\rho) \bigr] \dive(\bfu) = 0.
\end{equation}
Then, multiplying \eqref{eq:e_int} by $\varphi'_e(e)$ yields, once again formally, since $\varphi'_e(z) <0$ for $z >0$:
\begin{equation}\label{eq:ent_e}
\partial_t \bigl[\rho\, \varphi_e(e) \bigr] + \dive \bigl[ \rho \, \varphi_e(e)\, \bfu \bigr] + \varphi'_e(e)\, p \,\dive (\bfu) \leq 0.
\end{equation}
Summing \eqref{eq:ent_m} and \eqref{eq:ent_e} and noting that $\varphi_\rho$ and $\varphi_e$ are chosen such that
\begin{equation}\label{prop-entropie}
\rho\varphi'_\rho(\rho)-\varphi_\rho(\rho)+ \varphi'_e(e)\,p=0,
\end{equation}
we obtain \eqref{eq:entropy}, which is an entropy balance for the Euler equations, for the specific entropy defined by \eqref{eq:def_eta}.

\medskip
In the sequel we derive some analogous discrete entropy inequalities (with a possible remainder tending to 0) for the fully discrete, time semi-implicit (and fully implicit, \ie\ backward Euler, as far as System \eqref{eq:cont} only is concerned) or segregated schemes (fully explicit regarding System \eqref{eq:cont} only) presented in Section \ref{sec:schemes}, with a possible upwinding limited to that of the convection terms with respect to the material velocity.
Note that the entropy inequalities that we obtain here apply to both the staggered schemes \cite{her-14-ons,gra-16-unc,gas-18-mus} and to the colocated scheme \cite{her-19-cel} which is also based on the internal energy; indeed the entropy depends only on the mass and internal energy which are scalar unknowns located at the center of the (primal) cells in both schemes, so System \eqref{eq:cont} involves only equations posed on the primal mesh.

\medskip
Depending on the time and space discretization, we obtain three types of results:
\bli
\item local entropy estimates, {\it i.e.} discrete analogues of \eqref{eq:entropy}, in which case the scheme is entropy stable, 
\item global entropy estimates, {\it i.e.} discrete analogues of:
\begin{equation}
\frac d {dt} \int_\Omega \eta(\rho,e) \dx \leq 0; \label{eq:global-entropy}
\end{equation}
such a relation is a stability property of the scheme; this kind of relation was also proven in e.g. \cite{coq-06-sec} for a higher order scheme for the 1D Euler equations;
\item ``weak local" entropy inequalities, {\it i.e.} results of the form:
\[
\partial_t \eta(\rho,e) + \dive\bigl[ \eta(\rho,e)\, \bfu \bigr] +\widetilde{\mathcal{R}} \leq 0,  
\]
with $\widetilde{\mathcal{R}}$ tending to zero in a suitable sense with respect to the space and time discretization steps (or combination of both parameters), provided that the approximate solutions are controlled in reasonable norms, here, $L^\infty$ and BV norms.
Then a "Lax-consistency" property holds, of the form: a limit $(\bar \rho, \bar \bfu, \bar e)$ of a convergent subsequence of  approximate solutions  given by  the considered numerical scheme and bounded in the $L^\infty$ and BV norms, satisfies the following weak entropy inequality:
\begin{multline}
-\int_0^T \int_\Omega  \eta(\bar \rho, \bar e)\, \partial_t \varphi + \eta(\bar \rho,\bar e)\, \bfu \cdot \gradi \varphi \dx \dt
- \int_\Omega  \eta(\bar \rho,\bar e)(\bfx,0)\ \varphi(\bfx,0) \dx \leq 0,
\\
\mbox{for any function }\varphi \in \xC^\infty_c \bigl([0,T)\times \bar\Omega\bigr), \varphi \geq 0.
\label{eq:weak-entropy}
\end{multline}
\end{list} 

\medskip
In the sequel we address implicit schemes (Section \ref{sec:implicit}) and segregated explicit schemes (Section \ref{sec:explicit}).
For implicit schemes, we first consider an upwind discretization for which we get a local discrete entropy inequality (Theorem \ref{thrm:impl_upw}), and then a MUSCL-like improvement of the discretization of the convection term in order to reduce the numerical diffusion, for which we only get a global entropy estimate and a weak local entropy inequality (Theorem \ref{thrm:impl_muscl}). 
The case of explicit schemes is a little more tricky: we again consider the same two discretizations ({\it i.e.} upwind and MUSCL-like)  but we first deal with the mass balance equation, then with the internal energy equation, and combine the results to address entropy inequalities.

%
%
\subsection{Meshes and discrete norms}

Let  $\mesh$ be a mesh of the domain $\Omega$, supposed to be regular in the usual sense of the finite element literature (see \textit{e.g.} \cite{cia-91-bas}).
By $\edges$ and $\edges(K)$ we denote the set of all $(d-1)$-faces $\edge$ of the mesh and of the cell $K \in \mesh$ respectively, and we suppose that the number of the faces of a cell is bounded.
The set of faces included in $\Omega$ (resp. in the boundary $\partial \Omega$) is denoted by $\edgesint$ (resp. $\edgesext$); a face $\edge \in \edgesint$ separating the cells $K$ and $L$ is denoted by $\edge=K|L$.
For $K \in \mesh$ and $\edge \in \edges$, we denote by $|K|$ the measure of $K$ and by $|\edge|$ the $(d-1)$-measure of the face $\edge$.
The following quantities related to the mesh are used in the sequel:\begin{align}\label{eq:def_h_bar}
&\displaystyle h_\mesh=\max_{K\in\mesh} h_K \mbox{ with } h_K={\rm diam}(K), & & \underline h_\mesh = \min_{K \in \mesh} \frac{|K|}{\displaystyle \sum_{\edge \in \edges(K)} |\edge|}.
\\ \label{eq:def_Cm}
&\displaystyle  C_\mesh = \max_{K\in \mesh,\ (\edge,\edge')\in \edges(K)^2}\ \frac{(|\edge|+|\edge'|)\,h_K}{|K|},& &f_\mesh= \max_{K\in \mesh} \mbox{card } \edges(K).
\end{align}

Let $(t_n)_{0\leq n \leq N}$, with $0=t_0 < t_1 <\ldots < t_N=T$, define a partition of the time interval $(0,T)$, which we suppose uniform for the sake of simplicity, and let $\delta t=t_{n+1}-t_n$ for $0 \leq n \leq N-1$ be the (constant) time step.

\medskip
The discrete pressure, density and the internal energy unknowns are associated with the cells of the mesh $\mesh$; they are denoted by:
\[
\big\{ p^n_K,\ \rho^n_K,\ e^n_K,\ K \in \mesh,\ 0 \leq n \leq N \big\}.
\]

In the estimates given below, we shall need some discrete norms that we now define.

\begin{definition}[Discrete BV semi-norms and weak $L^1(0,T;(W^{1,+\infty}_0)')$ norm]
For a family $(z_K^n)_{K\in\mesh, 0 \leq n \leq N}\subset \xR$, let us define the following norms of the associated piecewise constant function $z$:
\begin{equation}
\begin{array}{l} \displaystyle
\normtbv{z}= \sum_{n=0}^N\ \sum_{K \in \mesh} |K|\ |z^{n+1}_K-z^n_K|,\label{def:time-BV}
\\ \displaystyle
\normxbv{z}= \sum_{n=0}^N \delta t \sum_{\edge=K|L \in \edgesint} |\edge|\ |z^n_L-z^n_K|,
\\ \displaystyle
\norm{z}_{-1,1,\star}  = \sup_{\displaystyle \psi \in \xC^\infty_c([0,T)\times \Omega)} \quad
\frac 1 {\norm{\gradi \psi}_\infty}\ \Bigl[\sum_{n=0}^N \delta t \sum_{K \in \mesh}|K|\ z_K^n \psi_K^n\Bigr],
\end{array}
\end{equation}
where $\psi_K^n$ stands for $\psi(\bfx_K,t_n)$, with $\bfx_K$ the mass center of $K$. 
Note that this latter weak norm is the discrete equivalent of the continuous dual norm of $v \in L^1(\Omega \times (0,T))$, defined by
\[
\norm{v}_{L^1(0,T ;(W^{1,+\infty}_0)')} =  \sup_{\displaystyle \psi \in \xC^\infty_c([0,T)\times \Omega)} \quad
\frac 1 {\norm{\gradi \psi}_\infty}\ \int_0^T \int_\Omega v\, \psi \dx \dt.
\]
\end{definition}
Some of the proofs below are based on the following convexity result \cite[Lemma 2.3]{gal-08-unc}.
In its formulation, and throughout the paper, $\li a,\ b \ri$ stands for the interval $[\min(a,b),\ \max(a,b)]$, for any real numbers $a$ and $b$.

\begin{lemma}\label{lem:int_convexity}
Let $\varphi$ be a strictly convex and continuously differentiable function over an open interval $I$ of $\xR$.
Let $x_K \in I$ and $x_L \in I$ be two real numbers.
Then the relation
\begin{multline} \label{eq:conv_int}
\varphi(x_K) + \varphi'(x_K)\,(x_{KL}-x_K) = \varphi(x_L) + \varphi'(x_L)\,(x_{KL}-x_L ) \mbox{ if } x_K \neq x_L,
\\
x_{KL} = x_K=x_L \mbox{ otherwise}	
\end{multline}
uniquely defines the real number $x_{KL}$ in $\li x_K, x_L \ri$.  
\end{lemma}

\begin{remark}[$x_{KL}$ for $\varphi(z)=z^2$]\label{rmrk:xKL}
Let us consider the specific function $\varphi(z)=z^2$.
Then, an easy computation yields  $x_{KL}=(x_K + x_L) / 2$ \textit{i.e.} the centered approximation.
\end{remark}

%
\subsection{Implicit schemes}\label{sec:implicit}
\medskip
With the above notations, the space time discretization of System \eqref{eq:cont} reads:
\begin{subequations}\label{eq:impl}
\begin{align}
\nonumber
& \mbox{For } K \in \mesh,\ 0 \leq n \leq N-1, 
\\[1ex] 
\label{eq:mass_i} &
\ \frac{|K|}{\delta t} (\rho_K^{n+1}-\rho_K^n) + \sum_{\edge\in \edges(K)} F_{K,\edge}^{n+1} = 0,
\\[1ex] \label{eq:e_int_i} &
\ \frac{|K|}{\delta t} (\rho_K^{n+1}e_K^{n+1}-\rho_K^n e_K^n) + \sum_{\edge\in \edges(K)} F_{K,\edge}^{n+1} e_\edge^{n+1}
+ p_K^{n+1} \sum_{\edge\in \edges(K)} |\edge|\,u_{K,\edge}^{n+1} \geq 0,
\\[1ex] \label{eq:etat_i} &
\ p_K^{n+1}=(\gamma-1)\, \rho_K^{n+1}\, e_K^{n+1},
\end{align}\end{subequations}
where $F_{K,\edge}^{n+1}$ is the mass flux through the face $\edge$, $e_\edge^{n+1}$ is an approximation of the internal energy at the face $\edge$, and $u_{K,\edge}^{n+1}$ stands for an approximation of the normal velocity to the face $\edge$; note that the velocity is solved in the full scheme by a space discretization of the momentum prediction and correction equations \eqref{eq:sch_pred}-\eqref{eq:sch_corr}.
Consistently with the boundary conditions, $u_{K,\edge}^{n+1}$ vanishes on every external face.
The mass flux $F_{K,\edge}^{n+1}$ reads:
\begin{equation} \label{mass-flux}
F_{K,\edge}^{n+1}=|\edge|\ \rho_\edge^{n+1} u_{K,\edge}^{n+1},
\end{equation}
where $\rho_\edge^{n+1}$ stands for an approximation of the density on $\edge$.
Throughout the paper, we suppose that $\rho_K^n$, $e_K^n$, $\rho_\edge^n$ and $e_\edge^n$ are positive, for any $K\in\mesh$, $\edge\in\edgesint$, $0\leq n \leq N$, which is verified by the solutions of the schemes presented in \cite{her-14-ons,her-18-cons,gas-18-mus,gou-17-sta} (of course, with positive initial conditions for $\rho$ and $e$).

\medskip
The two following lemmas are straigthforward consequences of Lemmas A1 and A2 in \cite{her-14-ons} and state discrete analogues of \eqref{eq:ent_m} and \eqref{eq:ent_e} respectively which are used to obtain the entropy inequalities.

\begin{lemma}\label{lem:mass}
Let $K \in \mesh$, $n$ be such that $0\leq n \leq N-1$ and let us suppose that the discrete mass balance \eqref{eq:mass_i} holds.
Let $\varphi$ be a twice continuously differentiable function defined over $(0,+\infty)$.
Then 
\begin{multline}\label{eq:R_mass} 
\frac{|K|}{\delta t} \Bigl[\varphi(\rho_K^{n+1})-\varphi(\rho_K^n)\Bigr]
+ \sum_{\edge\in \edges(K)} |\edge|\ \varphi(\rho_\edge^{n+1})\, u_{K,\edge}^{n+1} \\
+ \Bigl[\rho_K^{n+1} \varphi'(\rho_K^{n+1}) - \varphi(\rho_K^{n+1}) \Bigr] \sum_{\edge\in \edges(K)} |\edge|\  u_{K,\edge}^{n+1}
+|K|\, (R_m)_K^{n+1}= 0,
\\
\mbox{with }|K|\,(R_m)_K^{n+1}= \frac 1 2 \frac{|K|}{\delta t}\ \varphi''(\rho_K^{n+1/2})\ (\rho_K^{n+1}-\rho_K^n)^2
\hspace{20ex} \\
+ \sum_{\edge\in \edges(K)} |\edge|
\ \Bigl[\varphi(\rho_K^{n+1}) - \varphi(\rho_\edge^{n+1}) + \varphi'(\rho_K^{n+1}) (\rho_\edge^{n+1}-\rho_K^{n+1})\Bigr] u_{K,\edge}^{n+1}, 
\end{multline}
where $\rho_K^{n+1/2}\in \li \rho_K^n ,\rho_K^{n+1} \ri.$
\end{lemma}

\begin{lemma}\label{lem:energy}
Let $K \in \mesh$ and $n$ be such that $0\leq n \leq N-1$.
Let $\varphi$ be a twice continuously differentiable function defined over $(0,+\infty)$.
Then:
\begin{multline}\label{eq:R_e}
\varphi'(e_K^{n+1})\ \Bigl[ \frac{|K|}{\delta t} (\rho_K^{n+1}e_K^{n+1}-\rho_K^n e_K^n) + \sum_{\edge\in \edges(K)} F_{K,\edge}^{n+1} e_\edge^{n+1} \Bigr]
= \\
\frac{|K|}{\delta t} \Bigl[\rho_K^{n+1} \varphi(e_K^{n+1})-\rho_K^n \varphi(e_K^n) \Bigr]
+ \sum_{\edge\in \edges(K)} F_{K,\edge}^{n+1}\, \varphi(e_\edge^{n+1}) + |K|\,(R_e)_K^{n+1},
\\
\mbox{with } \; |K|\,(R_e)_K^{n+1}=\frac 1 2 \frac{|K|}{\delta t} \rho^n_K\ \varphi''(e_K^{n+1/2})(e_K^{n+1}-e_K^n)^2 \hspace{20ex} 
\\+\sum_{\edge\in \edges(K)} F_{K,\edge}^{n+1}\ \Bigl[\varphi(e_K^{n+1}) - \varphi(e_\edge^{n+1}) + \varphi'(e_K^{n+1}) (e_\edge^{n+1}-e_K^{n+1}) \Bigr],
\end{multline}
where $e_K^{n+1/2} \in \li e_K^n, e_K^{n+1} \ri$.
\end{lemma}
%
%
\subsubsection{Upwind implicit schemes --}
In this section, we suppose that the convection fluxes are approximated with a first order upwind scheme, {\it i.e.}, for $\edge \in \edgesint$, $\edge=K|L$, $\rho_\edge^{n+1}=\rho_K^{n+1}$ and $e_\edge^{n+1}=e_K^{n+1}$ if $u_{K,\edge} \geq 0$, $\rho_\edge^{n+1}=\rho_L^{n+1}$ and $e_\edge^{n+1}=e_L^{n+1}$ otherwise.
In this case, the scheme \eqref{eq:impl} satisfies a local entropy estimate ({\it i.e.} a discrete analogue of Inequality \eqref{eq:entropy}) which is stated in Theorem \ref{thrm:impl_upw} below.
Of course, this local entropy inequality also yields the global discrete inequality analogue to \eqref{eq:global-entropy}; furthermore, passing to the limit on the upwind implicit (or pressure correction) scheme applied the Euler equations, this local estimate also yields the Lax consistency, {\it i.e.} any limit $(\bar \rho, \bar \bfu, \bar e)$ of a convergent subsequence of  approximate solutions satisfies the  weak entropy inequality \eqref{eq:weak-entropy}.

\begin{theorem}[Discrete entropy inequality, implicit upwind scheme]\label{thrm:impl_upw}
Let $\eta$ be defined by \eqref{eq:def_eta}, and, for $0\leq n\leq N-1$, let $\eta_K^m=\eta(\rho_K^m,e_K^m)$ for $m=n,\ n+1$ and $K \in \mesh$, and $\eta_\edge^{n+1}=\eta(\rho_\edge^{n+1},e_\edge^{n+1})$ for $\edge \in \edgesint$.
Then any solution of the scheme \eqref{eq:impl} satisfies, for any $K\in\mesh$ and $0 \leq n \leq N-1$:
\[
\frac{|K|}{\delta t} (\eta_K^{n+1}-\eta_K^n)
+ \sum_{\edge\in \edges(K)} |\edge|\ \eta_\edge^{n+1} u_{K,\edge}^{n+1} \leq 0.
\]
\end{theorem}

\begin{proof}
Let $\varphi$ be a twice continuously differentiable function.
By Lemma \ref{lem:mass}, we get that \eqref{eq:R_mass} holds. 
For $\edge \in \edgesext$, thanks to the boundary conditions, the convection fluxes vanish.
For $0 \leq n \leq N-1$ and $K\in \mesh$, consider the term $(T_m)_{K,\edge}^{n+1}$ associated to an internal face $\edge=K|L$ in the remainder term $(R_m)_K^{n+1}$:
\begin{align*}
(T_m)_{K,\edge}^{n+1}& =\Bigl[\varphi(\rho_K^{n+1}) - \varphi(\rho_\edge^{n+1}) + \varphi'(\rho_K^{n+1}) (\rho_\edge^{n+1}-\rho_K^{n+1})\Bigr] u_{K,\edge}^{n+1}\\
& =
-\frac 1 2 \varphi''(\rho_{\edge,K}^{n+1})\ (\rho_\edge^{n+1}-\rho_K^{n+1})^2 u_{K,\edge}^{n+1},
\end{align*}
where $\rho_{\edge,K}^{n+1} \in \li \rho_\edge^{n+1}, \rho_K^{n+1} \ri$.
With the upwind choice, if $u_{K,\edge}^{n+1} \geq 0$, $\rho_\edge^{n+1}=\rho_K^{n+1}$ and $(T_m)_{K,\edge}^{n+1}$ vanishes.
If $u_{K,\edge}^{n+1} < 0$ and $\varphi''$ is a non-negative function ({\it i.e.} $\varphi$ is convex), $(T_m)_{K,\edge}^{n+1}$ is non-negative and so is $(R_m)_K^{n+1}$, for any $K \in \mesh$.
Since $\varphi_\rho$ defined by \eqref{eq:vphis} is indeed convex, Lemma \ref{lem:mass} implies that any solution $\{\rho_K^n, K\in \mesh, 0 \leq n \leq N \}$ to Equation \eqref{eq:mass_i} of the scheme satisfies, for $K \in \mesh$ and $0 \leq n \leq N-1$:
\begin{multline}\label{eq:ent_mi}
\frac{|K|}{\delta t} \Bigl[\varphi_\rho(\rho_K^{n+1})-\varphi_\rho(\rho_K^n)\Bigr]
+ \sum_{\edge\in \edges(K)} |\edge|\ \varphi_\rho(\rho_\edge^{n+1}) u_{K,\edge}^{n+1}
\\
+ \Bigl[\rho_K^{n+1} \varphi_\rho'(\rho_K^{n+1}) - \varphi_\rho(\rho_K^{n+1}) \Bigr] \sum_{\edge\in \edges(K)} |\edge|\  u_{K,\edge}^{n+1}
\leq 0.
\end{multline}
Now turning to Lemma \ref{lem:energy}, by similar arguments,  the remainder term $(R_e)_K^{n+1}$ in \eqref{eq:R_e} is nonnegative for any regular convex function $\varphi$, for any $K\in\mesh$ and $0 \leq n \leq N-1$.
Hence, since $\varphi_e$ defined by Equation \eqref{eq:vphis} is convex, we get that any solution to \eqref{eq:e_int_i} satisfies:
\begin{multline}\label{eq:ent_ei}
\frac{|K|}{\delta t} \Bigl[\rho_K^{n+1} \varphi_e(e_K^{n+1})-\rho_K^n \varphi_e(e_K^n) \Bigr]
\\+ \sum_{\edge\in \edges(K)} F_{K,\edge}^{n+1}\, \varphi_e(e_\edge^{n+1})
+ \varphi'_e(e_K^{n+1})\,p_K^{n+1} \sum_{\edge\in \edges(K)} |\edge|\  u_{K,\edge}^{n+1}
\leq 0.
\end{multline}
The desired relation is then obtained by summing the inequalities \eqref{eq:ent_mi} and \eqref{eq:ent_ei}, using \eqref{prop-entropie}.
\end{proof}
%
%
\subsubsection{MUSCL-like schemes --}
The aim of this section is to improve the approximation of the convection fluxes in \eqref{eq:mass_i} and \eqref{eq:e_int_i} in order to reduce the numerical diffusion, while still satisfying an entropy inequality.
This leads to a condition similar to the limitation procedure which is the core of a MUSCL procedure \cite{van-79-tow}; indeed, in order to yield an entropy inequality (instead of, for a MUSCL technique, to yield a maximum principle), the approximation of the unknowns at the face must be "sufficiently close to" the upwind approximation.
The entropy inequality is then obtained only in the weak sense.
The technique to reach this result consists in splitting the remainder terms appearing in Lemma \ref{lem:mass} and \ref{lem:energy} in two parts: the first one is non-negative under some condition for the face approximation (hence the above mentioned limitation requirement); the second one is conservative and can be bounded in a discrete negative Sobolev norm (this explains why the entropy estimate is only a weak one).

\medskip
Let $\varphi_\rho$ and $\varphi_e$ be the functions defined by \eqref{eq:vphis} and let $\edge\in\edgesint$, $\edge=K|L$; by Lemma \ref{lem:int_convexity}, there exists a unique $\rho_{KL}^{n+1} \in \li \rho_K^{n+1}, \rho_L^{n+1} \ri$ and $e_{KL}^{n+1}\in \li e_K^{n+1}, e_L^{n+1} \ri$ such that
\begin{subequations}
\begin{align}
&\varphi_\rho(\rho_K^{n+1})\!+\!\varphi'_\rho(\rho_K^{n+1}) \bigr[\rho_{KL}^{n+1}\!-\!\rho_K^{n+1}\bigl]
= \varphi_\rho(\rho_L^{n+1})\!+\!\varphi'_\rho(\rho_L^{n+1})\bigr[\rho_{KL}^{n+1}\!-\!\rho_L^{n+1}\bigl], \label{rhoKL}
\\
&\varphi_e(e_K^{n+1})\!+\!\varphi'_e(e_K^{n+1})\ \bigr[e_{KL}^{n+1}\!-\!e_K^{n+1}\bigl]
= \varphi_e(e_L^{n+1})\!+\!\varphi'_e(e_L^{n+1})\ \bigr[e_{KL}^{n+1}\!-\!e_L^{n+1}\bigl].\label{eKL}
\end{align}
\label{def:rhoeKL}
\end{subequations}
Entropy estimates are obtained in Theorem \ref{thrm:impl_muscl} under the following conditions:
\begin{subequations}\label{eq:H-imp}  
\begin{align} \label{eq:H-rho-imp} &
\rho_\edge^{n+1} \in \li \rho_K^{n+1}, \ \rho_{KL}^{n+1} \ri \mbox{ if } u_{K,\edge}^{n+1} \geq 0, \quad
\rho_\edge^{n+1} \in \li \rho_L^{n+1},\ \rho_{KL}^{n+1} \ri \mbox{ otherwise},  
\\ \label{eq:H-e-imp} &
e_\edge^{n+1}\in \li e_K^{n+1},\ e_{KL}^{n+1}\ri \mbox{ if } u_{K,\edge}^{n+1} \geq 0, \quad
e_\edge^{n+1}\in \li e_L^{n+1},\ e_{KL}^{n+1}\ri \mbox{ otherwise},
\end{align}
\end{subequations}
where $\rho_{KL}^{n+1}$ and $e_{KL}^{n+1}$ are defined by \eqref{def:rhoeKL}.
Note that these conditions are satisfied by the upwind scheme \eqref{eq:impl}.
They may be seen as an additional constraint to be added to the limitation of a MUSCL-like procedure (see also the conclusion of the last section of this paper).

\begin{theorem}[Entropy inequalities, implicit MUSCL-like scheme]\label{thrm:impl_muscl}
Let us assume that, for $\edge \in \edgesint$, $\edge=K|L$ and for $0 \leq n \leq N-1$, the approximate density $\rho_\edge^{n+1}$ and internal energy $e_\edge^{n+1}$ in the numerical mass fluxes \eqref{mass-flux} and in the internal energy balance \eqref{eq:e_int_i} satisfy the conditions \eqref{eq:H-imp}.

Then any solution of the scheme \eqref{eq:impl} satisfies, for any $K\in\mesh$ and $0 \leq n \leq N-1$:
\[
\frac{|K|}{\delta t} (\eta_K^{n+1}-\eta_K^n)
+ \sum_{\edge\in \edges(K)} |\edge|\ \eta_\edge^{n+1} u_{K,\edge}^{n+1} + |K|\ (\delta\!R_\eta)_K^{n+1} \leq 0,
\]
where the remainder term $\delta\!R_\eta$ satisfies $\sum_{K\in\mesh}|K|\ (\delta\!R_\eta)_K^{n+1} =0$ so that, integrating in space ({\it i.e.} summing over the cells), the following global discrete entropy estimate holds for $0 \leq n \leq N-1$:
\[
\sum_{K\in\mesh} |K|\ \eta_K^{n+1} \leq \sum_{K\in\mesh} |K|\ \eta_K^n.
\]
In addition, let us suppose that there exists $M >0$ such that $\rho_K^n \leq M$, $1/ \rho_K^n \leq M$, $e_K^n \leq M$, $1/e_K^n \leq M$ and $|u_{K,\edge}^n| \leq M$ for $K \in \mesh$, $\edge \in \edges(K)$ and $0 \leq n \leq N$, and let us define the quantities $|\varphi_\rho'|_\infty=\max(|\varphi'_\rho(1/M)|,\ |\varphi'_\rho(M)|)$ and $|\varphi_e'|_\infty=\max(|\varphi'_e(1/M)|,\ |\varphi'_e(M)|)$.
Then the remainder term $\delta\!R_m$ satisfies the following bound:
\begin{equation}
\norm{\delta \!R_m}_{-1,1,\star}  \leq 3\ M\ \bigl(|\varphi_\rho'|_\infty\ \normxbv{\rho}+M\ |\varphi_e'|_\infty\ \normxbv{e})\ h_\mesh.
\end{equation}
Therefore, a Lax-consistency property holds; more precisely, any limit $(\bar \rho, \bar \bfu, \bar e)$ of a converging sequence of approximate solutions bounded in the $L^\infty$ and BV norms satisfies \eqref{eq:weak-entropy}.
\end{theorem}

\begin{proof}  	
Let $(\delta \varphi_\rho)^{n+1}_\edge$ be defined by:
\begin{multline} \label{eq:def_drho}
(\delta \varphi_\rho)^{n+1}_\edge=\varphi_\rho(\rho_K^{n+1})-\varphi_\rho(\rho_\edge^{n+1})
+ \varphi'_\rho(\rho_K^{n+1})\ \bigr[\rho_{KL}^{n+1}-\rho_K^{n+1}\bigl]
\\
+ \frac 1 2\,\bigl[\varphi'_\rho(\rho_K^{n+1})+\varphi'_\rho(\rho_L^{n+1})\bigr]\ \bigl[\rho_\edge^{n+1}-\rho_{KL}^{n+1}\bigr].
\end{multline}
By Lemma \ref{lem:mass}, \eqref{eq:R_mass} holds; an easy computation shows that the term associated to the face $\edge$ in the expression of the remainder term $(R_m)_K^{n+1}$ satisfies:
\begin{align*}
(F_m)_{K,\edge}^{n+1} &=
|\edge|\,\Bigl[\varphi_\rho(\rho_K^{n+1}) - \varphi_\rho(\rho_\edge^{n+1}) + \varphi_\rho'(\rho_K^{n+1}) (\rho_\edge^{n+1}-\rho_K^{n+1})\Bigr] u_{K,\edge}^{n+1}
\\
& = |\edge|\,(\delta \varphi_\rho)^{n+1}_\edge\ u_{K,\edge}^{n+1} + (F^R_m)_{K,\edge}^{n+1}
\end{align*}
with $\displaystyle
(F^R_m)_{K,\edge}^{n+1} = |\edge|\,\frac 1 2\,\Bigl[\varphi'_\rho(\rho_K^{n+1})-\varphi'_\rho(\rho_L^{n+1})\Bigr]
\ (\rho_\edge^{n+1}-\rho_{KL}^{n+1})\ u_{K,\edge}^{n+1}.$
Thanks to the assumption \eqref{eq:H-rho-imp}, since $\varphi'_\rho$ is an increasing function, $(F^R_m)_{K,\edge}^{n+1}\geq 0$.
Let us define $(\delta\!R_m)_K^{n+1}$, $K \in \mesh$, $0 \leq n \leq N-1$ by:
\begin{equation}
|K|\,(\delta\!R_m)_K^{n+1} = \sum_{\edge\in\edges(K)} |\edge|\ (\delta \varphi_\rho)^{n+1}_\edge\ u_{K,\edge}^{n+1}.
\label{deltaRm}	
\end{equation}
Then, under assumption \eqref{eq:H-rho-imp}, we get:
\begin{multline}\label{eq:ent_mi_c}
\frac{|K|}{\delta t} \bigl[\varphi_\rho(\rho_K^{n+1})-\varphi_\rho(\rho_K^n)\bigr]
+ \sum_{\edge\in \edges(K)} |\edge|\ \varphi_\rho(\rho_\edge^{n+1}) u_{K,\edge}^{n+1}
\\
+ \Bigl[\rho_K^{n+1} \varphi_\rho'(\rho_K^{n+1}) - \varphi_\rho(\rho_K^{n+1}) \Bigr] \sum_{\edge\in \edges(K)} |\edge|\  u_{K,\edge}^{n+1}
+ |K|\ (\delta\!R_m)_K^{n+1}
\leq 0.
\end{multline}
Let us prove that $\delta\! R_m$  satisfies:
\begin{equation} \label{eq:deltaRm}
\norm{\delta\!R_m}_{-1,1,\star} \leq 3\, M\ |\varphi_\rho'|_\infty\ \normxbv{\rho}\ h_\mesh.
\end{equation}
Indeed, since both $\rho_\edge^{n+1}$ and $\rho_{KL}^{n+1}$ lie in the interval $\li \rho_K^{n+1},\ \rho_L^{n+1}\ri$, we have by convexity of $\varphi_\rho$:
\[
|(\delta \varphi_\rho)_\edge^{n+1}| \leq 3\, \max\bigl(|\varphi'_\rho(\rho_K^{n+1})|,\ |\varphi'_\rho(\rho_L^{n+1})|\bigr)\ |\rho_K^{n+1}-\rho_L^{n+1}|.
\]
Let $\psi$ be a function of $\xC^\infty_c(\Omega\times(0,T))$.
We have, thanks to the conservativity of the remainder term:
\begin{align*}
T&=
\sum_{n=0}^{N-1} \delta t \sum_{K \in \mesh}|K|\ (\delta\!R_m)_K^{n+1} \psi_K^{n+1} \\
&=
\sum_{n=0}^{N-1} \delta t \sum_{\edge=K|L \in \edgesint} |\edge|\ (\delta \varphi_\rho)_\edge^{n+1}\ (\psi_K^{n+1}-\psi_L^{n+1})\ u_{K,\edge}.
\end{align*}
Therefore,
\[
|T| \leq 3\, |\varphi'_\rho|_\infty\, M\ \bigl[  \norm{\gradi \psi}_\infty\bigr]\ h_\mesh
\sum_{n=0}^{N-1} \delta t \sum_{\edge=K|L \in \edgesint} |\edge|\ |\rho_K^{n+1}-\rho_L^{n+1}|,
\]
which concludes the proof of \eqref{eq:deltaRm}.

\medskip
Following the same line of thought for the internal energy balance, let  $(\delta \varphi_e)_\edge^{n+1}$ be defined by:
\begin{multline} \label{eq:def_de}
(\delta \varphi_e)^{n+1}_\edge=\varphi_e(e_K^{n+1})-\varphi_e(e_\edge^{n+1})
+ \varphi'_e(e_K^{n+1})\ \bigr[e_{KL}^{n+1}-e_K^{n+1}\bigl]
\\
+ \frac 1 2\,\bigl[\varphi'_e(e_K^{n+1})+\varphi'_e(e_L^{n+1})\bigr]\ \bigl[e_\edge^{n+1}-e_{KL}^{n+1}\bigr],
\end{multline}
and $(\delta \!R_e)_K^{n+1}$ the remainder term given by:
\begin{equation} \label{deltaRe}
|K|\,(\delta \!R_e)_K^{n+1} = \sum_{\edge\in\edges(K)}  (\delta \varphi_e)^{n+1}_\edge\ F_{K,\edge}^{n+1}.
\end{equation}
Thanks to the assumption \eqref{eq:H-e-imp} we get:
\begin{multline}\label{eq:ent_ei_c}  \frac{|K|}{\delta t} \Bigl[\rho_K^{n+1} \varphi_e(e_K^{n+1})-\rho_K^n \varphi_e(e_K^n)\Bigr]
+ \sum_{\edge\in \edges(K)} \varphi_e(e_\edge^{n+1}) F_{K,\edge}^{n+1}
\\
+ \varphi_e'(e_K^{n+1}) p_K^{n+1} \sum_{\edge\in \edges(K)} |\edge|\  u_{K,\edge}^{n+1}
+ |K|\ (\delta\!R_e)_K^{n+1}
\leq 0.
\end{multline}
In addition, $\delta\!R_e$ satisfies the following inequality:
\begin{equation}\label{eq:deltaRe}
\norm{\delta\!R_e}_{-1,1,\star}  \leq 3\, M^2\ |\varphi_e'|_\infty\ \normxbv{e}\ h_\mesh.
\end{equation} 
Combining the inequalities \eqref{eq:ent_mi_c} and \eqref{eq:ent_ei_c} and thanks to \eqref{prop-entropie},  \eqref{eq:deltaRm} and  \eqref{eq:deltaRe} concludes the proof of the theorem. 
\end{proof}
%
%
\subsection{Explicit schemes} \label{sec:explicit}

The general form of the discrete analogue of System \eqref{eq:cont} for an explicit scheme reads:
\begin{subequations}\label{eq:expl}
\begin{align}
\nonumber
& \mbox{For } K \in \mesh,\ 0 \leq n \leq N-1, 
\\[1ex] 
\label{eq:mass_e} &
\quad \frac{|K|}{\delta t} (\rho_K^{n+1}-\rho_K^n) + \sum_{\edge\in \edges(K)} F_{K,\edge}^n = 0,
\displaybreak[1] \\[1ex] 
\label{eq:e_int_e} &
\quad \frac{|K|}{\delta t} (\rho_K^{n+1}e_K^{n+1}-\rho_K^n e_K^n) + \sum_{\edge\in \edges(K)} F_{K,\edge}^n e_\edge^n
+ p_K^n \sum_{\edge\in \edges(K)} |\edge|\,u_{K,\edge}^n \geq 0,
\displaybreak[1] \\[1ex] \label{eq:etat_e} &
\quad p_K^n=(\gamma-1)\, \rho_K^n\, e_K^n,
\end{align}\end{subequations}
where the numerical mass flux $F_{K,\edge}^n$ is still defined by \eqref{mass-flux}.

\medskip
Let $\rho_{KL}^n$ (resp. $e_{KL}^n$) be the real number defined by Equation \eqref{eq:conv_int} with $x_K=\rho^n_K$ (resp. $x_K=e^n_K$) and $x_L=\rho^n_L$ (resp. $x_L=e^n_L$) and $\varphi=\varphi_\rho$ (resp. $\varphi=\varphi_e$) and let us assume that for $\edge \in \edgesint$, $\edge=K|L$ and for $0 \leq n \leq N-1$,
\begin{align} \label{eq:H_rho^exp} &
\rho_\edge^n \in \li \rho_K^n,\ \rho_{KL}^n \ri \mbox{ if } u_{K,\edge}^n \geq 0, \qquad \rho_\edge^n \in \li \rho_L^n,\ \rho_{KL}^n\ri \mbox{ otherwise.}
\\ \label{eq:H_e^exp} &
e_\edge^n \in \li e_K^n,\ e_{KL}^n \ri \mbox{ if } u_{K,\edge}^n \geq 0,\qquad e_\edge^n \in \li e_L^n,\ e_{KL}^n\ri \mbox{ otherwise.} 
\end{align}

With these two conditions, Theorem \ref{theo:entropie} below yields a weak discrete entropy inequality, in the sense that a remainder term exists which tends to 0 (under some conditions) with the mesh and time steps, but its sign is unknown.
However, in the case of an upwind approximation of the density and the internal energy on the faces of the mesh (note that \eqref{eq:H_rho^exp} and  \eqref{eq:H_e^exp} are satisfied for such approximations), a local discrete entropy inequality can be obtained under the following additional conditions. 
\begin{enumerate}
\item First, the normal face velocities $u_{K,\sigma}$ in the mass flux \eqref{mass-flux} are assumed to be either 
\bli
\item computed from a discrete velocity field $\bfu$:
\begin{equation} \label{hyp-velocity}
u_{K,\edge} = \bfu_\edge \cdot \bfn_{K,\edge},
\end{equation}
where $\bfn_{K,\edge}$ is the unit normal vector to $\edge$ outward $K$ and $\bfu_\edge$ is an approximation of the velocity at the face, which may be the discrete unknown itself (when the velocity degrees of freedom are those of a non conforming Crouzeix-Raviart or Rannacher-Turek approximation, see e.g. \cite{gas-18-mus}) or an interpolation (for instance, for a colocated arrangement of the unknowns, as in \cite{her-19-cel}). 
\item the unknown themselves in the case of the staggered MAC scheme, since only the normal velocity is approximated in this case, see e.g. \cite{gas-18-mus}.
\end{list}
  
For $1 \leq r$, we then define the following discrete norm:
\begin{equation} \label{normq}
\norm{\bfu}_{L^r(0,T;W^{1,r}_\mesh)}^r = \sum_{i=1}^d\sum_{n=0}^N \delta t \sum_{K\in\mesh}
\ \sum_{(\edge,\edge')\in \edges^{(i)}(K)^2} |K| \left(\frac{u_{\edge,i}^n-u_{\edge',i}^n}{h_K} \right)^r,
\end{equation}
where $\edges^{(i)}(K)=\edges(K)$ for the Crouzeix-Raviart or Rannacher-Turek case and $\edges^{(i)}(K)$ is restricted to the two faces of $K$ perpendicular to the $i^{th}$ vector of the canonical basis of $\xR^d$ in the case of the MAC scheme. 

\begin{remark}[Discrete $L^r(W^{1,r})$ norm of the velocity]
It is reasonable to suppose that, under regularity assumptions of the mesh whose precise statement depends the space approximation at hand, this norm is equivalent to the standard finite-volume discrete $L^r(0,T; W^{1,r})$ norm \cite{eym-00-fin}; it is indeed true for usual cells (in particular, with a bounded number of faces) for staggered discretizations and for a convex interpolation of the velocity at the faces for colocated schemes.	
\end{remark}

\item Second, the following CFL conditions hold:
\begin{align} \label{eq:cfl_rho} &
\delta t \leq \frac{|K|}{\displaystyle
\sum_{\edge\in \edges(K)} \frac{\varphi_\rho''(\tilde \rho_K^{n+1/2})^2}{\varphi_\rho''(\rho_{K,\edge}^n)}\ |\edge|\ (u_{K,\edge}^n)^-},
\\[1ex] \label{eq:cfl_e} &
\delta t \leq \frac{\varphi_e''(e_K^{n+1/2})\ |K|\ \rho_K^{n+1}}
{\displaystyle \sum_{\edge\in \edges(K)} \frac{\varphi_e''(\tilde e_K^{n+1/2})^2}{\varphi_e''(e_{K,\edge}^n)}\ (F_{K,\edge}^n)^-}.
\end{align}
where $\tilde \rho_K^{n+1/2} \in \li \rho_K^n, \rho_K^{n+1}\ri$, $\rho_{K,\edge}^n \in \li \rho_K^n,\ \rho_L^n \ri$, $\tilde e_K^{n+1/2} \in \li e_K^n, e_K^{n+1}\ri$ and $e_{K,\edge}^n \in \li e_K^n,\ e_L^n \ri$ are defined by:
\begin{align}\label{tilderho-exp} &
\varphi_\rho''(\tilde \rho_K^{n+1/2})\,(\rho_K^{n+1}-\rho_K^n)^2 =\varphi_\rho'(\rho_K^{n+1})-\varphi_\rho'(\rho_K^{n}),
\\ \label{rhoKedge-exp} &
\varphi_\rho''(\rho_{K,\edge}^n)\ \bigl(\rho_K^n - \rho_L^n \bigr)^2 = \varphi(\rho_L^{n}) - \varphi_\rho(\rho_K^{n}) - \varphi_\rho'(\rho_K^{n}) \bigl(\rho_K^n - \rho_L^n \bigr),
\\ \label{tilde-e-exp} &
\varphi_e''(\tilde e_K^{n+1/2})\,(e_K^{n+1}-e_K^n)^2 =\varphi_e'(e_K^{n+1})-\varphi_e'(e_K^{n}),
\\ \label{eKedge-exp} &
\varphi_e''(e_{K,\edge}^n)\ \bigl(e_K^n - e_L^n \bigr)^2 = \varphi_e(e_L^{n}) - \varphi_e(e_K^{n}) - \varphi_e'(e_K^{n}) \bigl(e_K^n - e_L^n \bigr).
\end{align}
\end{enumerate}

\begin{theorem}[Discrete entropy inequalities, explicit schemes]\label{theo:entropie}
Let $\rho$ and $e$ satisfy the relations of the scheme \eqref{eq:expl}.
Let $M \geq 1$ and let us suppose that $\rho_K^n \leq M$, $1/ \rho_K^n \leq M$, $e_K^n \leq M$, $1/e_K^n \leq M$ and $|u_{K,\edge}| \leq M$, for $K \in \mesh$, $\edge \in \edges(K)$ and $0 \leq n \leq N$.
Assume that the discretization of the convection term in \eqref{eq:mass_e} and \eqref{eq:e_int_e} satisfies the assumptions \eqref{eq:H_rho^exp} and \eqref{eq:H_e^exp} respectively.
Let $\eta$ be defined by \eqref{eq:def_eta}.
Then any solution of the scheme \eqref{eq:expl} satisfies, for any $K\in\mesh$ and $0 \leq n \leq N-1$:
\[
\frac{|K|}{\delta t} (\eta_K^{n+1}-\eta_K^n)
+ \sum_{\edge\in \edges(K)} |\edge|\ \eta_\edge^n u_{K,\edge}^n + |K|\ (R_\eta)_K^n \leq 0,
\]
where  $R_\eta =R_{\eta,1}+R_{\eta,2}$ with:
\begin{align*} &
\norm{R_{\eta,1}}_{-1,1,\star}  \leq 3M\, \Bigl(|\varphi'_\rho|_\infty\,\normxbv{\rho}+M\ |\varphi'_e|_\infty\,\normxbv{e}\Bigr)\ h_\mesh,
\\ &
\norm{R_{\eta,2}}_{L^1} \leq M^2 \ \Bigl(|\varphi''_\rho|_\infty\,\normtbv{\rho}+|\varphi''_e|_\infty\,\normtbv{e}\Bigr)
\ \frac{\delta t}{\underline h_\mesh},
\end{align*}
where $\underline h_\mesh$ is defined by \eqref{eq:def_h_bar},   $|\varphi'_\rho|_\infty=\max(|\varphi'_\rho(1/M)|,\ |\varphi'_\rho(M)|)$, $|\varphi'_e|_\infty=\max(|\varphi'_e(1/M)|,\ |\varphi'_e(M)|)$, and  $|\varphi''_\rho|_\infty$ and $|\varphi''_e|_\infty$  denote the maximum value taken by $\varphi''_\rho$ and $\varphi''_e$ respectively on the interval $[1/M,\ M]$.

\medskip 
Moreover, if the discretization of the convection term in \eqref{eq:mass_e} and \eqref{eq:e_int_e} is upwind, and if the normal face velocities $u_{K,\edge}$ satisfy \eqref{hyp-velocity}, under the CFL conditions \eqref{eq:cfl_rho} and \eqref{eq:cfl_e}, we also have (with a different expression for $R_\eta$):
\begin{equation}\label{eq:R_ent_e}
\norm{R_\eta}_{L^1} \leq f_\mesh \ C_\mesh\ M^{(2q-1)/q}\ |\varphi''_\rho|_\infty\ \normtbv{\rho}^{\frac 1 q}\ \norm{\bfu}_{L^{q'}(0,T;W^{1,q'}_\mesh)}  \delta t^{\frac 1 q}.
\end{equation}
where $q \geq 1$, $q'\geq 1$ and $\dfrac 1 q + \dfrac 1 {q'}=1$, $f_\mesh$ and  $C_\mesh$ are defined by \eqref{eq:def_Cm}.
\end{theorem}
\begin{proof}
The results are obtained by applying the propositions \ref{prop:rho} and \ref{prop:e} below with $\varphi=\varphi_\rho$ and $\varphi=\varphi_e$ respectively.
\end{proof}
%
%

\medskip
The aim of the following proposition is to derive a discrete analogue of Relation \eqref{eq:ent_m}. 

\begin{proposition}[Discrete renormalized forms of the mass balance equation]\label{prop:rho}
Let $\varphi$ be a twice continuously differentiable convex function from $(0,+\infty)$ to $\xR$, and let $\rho$ satisfy \eqref{eq:mass_e}.
Let $M \geq 1$ and let us suppose that $\rho_K^n \leq M$, $1/ \rho_K^n \leq M$ and $|u_{K,\edge}| \leq M$, for $K \in \mesh$, $\edge \in \edges(K)$ and $0 \leq n \leq N$.
Let $|\varphi'|_\infty=\max(|\varphi'(1/M)|,\ |\varphi'(M)|)$ and $|\varphi''|_\infty$ be the maximum value taken by $\varphi''$ on the interval $[1/M,\ M]$.
Assume that $\rho_\edge^n $ satisfies \eqref{eq:H_rho^exp}.
Then the following inequality holds:
\begin{multline}
\frac{|K|}{\delta t} \bigl[\varphi(\rho_K^{n+1})-\varphi(\rho_K^n)\bigr]
+ \sum_{\edge\in \edges(K)} |\edge| \varphi(\rho_\edge^n) u_{K,\edge}^n
\\ + \bigl(\varphi'(\rho_K^n) \rho_K^n - \varphi(\rho_K^n)\bigr)\ \bigl[\sum_{\edge\in \edges(K)} |\edge|\, u_{K,\edge}^n\bigr]
+ |K|\,(R_\rho)_K^{n+1} \leq 0,
\label{est:rhoexp}
\end{multline}
where the remainder $(R_\rho)_K^{n+1}=(R_{\rho,1})_K^{n+1}+(R_{\rho,2})_K^{n+1}$ with:
\begin{align*}
&\norm{R_{\rho,1}}_{-1,1,\star}  \leq 3M\ |\varphi'|_\infty\ \normxbv{\rho}\ h_\mesh, \\
&\norm{R_{\rho,2}}_{L^1} \leq M^2\ |\varphi''|_\infty\ \normtbv{\rho}\ \frac{\delta t}{\underline h_\mesh},
\end{align*} 	
where $\underline h_\mesh$ is defined by \eqref{eq:def_h_bar}.

\medskip
Assume furthermore that the normal face velocities $u_{K,\edge}$ satisfy \eqref{hyp-velocity}, that the discretization of the convection term in \eqref{eq:mass_e} is upwind and that the CFL condition \eqref{eq:cfl_rho} holds with $\varphi$ instead of $\varphi_\rho$.
Then \eqref{est:rhoexp} still holds (with a different expression for $R_\rho$) and:
\[
\norm{R_\rho}_{L^1} \leq f_\mesh \ C_\mesh\ M^{(2q-1)/q}\ |\varphi''|_\infty\ \normtbv{\rho}^{1/q}\ \norm{\bfu}_{L^{q'}(0,T;W^{1,q'}_\mesh)}\ \delta t^{1/q},
\]
where $q \geq 1$, $q'\geq 1$, $\dfrac 1 q + \dfrac 1 {q'}=1$, $\norm{\cdot}_{L^{q'}(0,T;W^{1,q'}_\mesh)}$ is defined by \eqref{normq},  $f_\mesh$ and $C_\mesh$ are defined by \eqref{eq:def_Cm} and $C$ only depends on the maximal number of faces of the mesh cells.
\end{proposition}

\begin{proof}
Mimicking the formal computation performed at the continuous level, let us multiply \eqref{eq:mass_e} by $\varphi'(\rho_K^{n+1})$.
We get:
\[
\varphi'(\rho_K^{n+1}) \Bigl[\frac{|K|}{\delta t} (\rho_K^{n+1}-\rho_K^n)+ \sum_{\edge\in \edges(K)} F_{K,\edge}^n \Bigr] =
(T_1)_K^{n+1}+(T_2)_K^{n+1}+ |K|\,R_K^{n+1}=0,
\]
with
\begin{align} \nonumber &
(T_1)_K^{n+1}=\varphi'(\rho_K^{n+1})\ \frac{|K|}{\delta t} (\rho_K^{n+1}-\rho_K^n),
\quad (T_2)_K^{n+1}= \varphi'(\rho_K^n)\ \sum_{\edge\in \edges(K)} F_{K,\edge}^n,
\\ \label{eq:def_ren} &
|K|\,R_K^{n+1}=\bigl(\varphi'(\rho_K^{n+1})-\varphi'(\rho_K^n)\bigr)\ \sum_{\edge\in \edges(K)} F_{K,\edge}^n.
\end{align}
By a Taylor expansion, there exists $\rho_K^{n+1/2} \in \li\rho_K^n,\ \rho_K^{n+1}\ri$ such that:
\begin{multline}\label{eq:R_1}
(T_1)_K^{n+1}= \frac{|K|}{\delta t} \bigl[\varphi(\rho_K^{n+1})-\varphi(\rho_K^n)\bigr]+|K|\,(R_1)_K^{n+1}, 
\\
\mbox{with } (R_1)_K^{n+1}= \frac 1 {2 \delta t} \,\varphi''(\rho_K^{n+1/2})\,(\rho_K^{n+1}-\rho_K^n)^2 \geq 0.
\end{multline}
The term $(T_2)_K^{n+1}$ reads:
\begin{multline}
(T_2)_K^{n+1} = \sum_{\edge\in \edges(K)} |\edge|\, \varphi(\rho_\edge^n) u_{K,\edge}^n
\\
+ \bigl(\varphi'(\rho_K^n) \rho_K^n - \varphi(\rho_K^n)\bigr)\ \sum_{\edge\in \edges(K)} |\edge|  u_{K,\edge}^n
+ |K|\,(R_2)_K^{n+1},
\end{multline}
with
\[
|K|\,(R_2)_K^{n+1} =  \sum_{\edge\in \edges(K)} |\edge| 
\Bigl[\varphi(\rho_K^n) + \varphi'(\rho_K^n) (\rho_\edge^n - \rho_K^n) - \varphi(\rho_\edge^n) \Bigr] u_{K,\edge}^n. \label{R2exp}
\]
Thanks to assumption \eqref{eq:H_rho^exp}, the remainder $R_2$ is a sum of a non-negative part and a term tending to zero; indeed there exists $\delta\! R_2$ such that:
\[
R_2 \geq \delta\! R_2 \mbox{ and }
\norm{\delta\! R_2}_{-1,1,\star}  \leq 3M\ |\varphi'|_\infty\ \normxbv{\rho}\ h_\mesh.
\]
 This result is obtained by adapting the proof of the implicit case (indeed, up to a change of time exponents at the right-hand side from $n$ to $n+1$, the expression of $(R_2)_K^{n+1}$ is the same than the second term of $(R_m)_K^{n+1}$ in the expression \eqref{eq:R_mass}, and the computation from Relation \eqref{eq:def_drho} up to the end of the proof of \eqref{eq:deltaRm} may be reproduced, still with the same change of time exponents).

\medskip
Let us now prove that the remainder term $R=(R_K^n)_{K\in \mesh}^{n=0,\ldots,M}$ defined by \eqref{eq:def_ren} satisfies:
\begin{equation}
\norm{R}_{L^1} =\sum_{n=0}^{N-1} \delta t \sum_{K\in\mesh} |K|\, R_K^{n+1}  \leq M^2 \ |\varphi''|_\infty\ \normtbv{\rho}\ \frac{\delta t}{\underline h_\mesh}.
\label{est:R}
\end{equation}
Indeed, for $K \in \mesh$ and $0 \leq n \leq N$, we get:
\begin{align}\label{eq:def_R}
|K|\, R_K^{n+1} &=\bigl(\varphi'(\rho_K^{n+1})-\varphi'(\rho_K^n)\bigr)\ \sum_{\edge\in \edges(K)} F_{K,\edge}^n
\\ &= \varphi''(\tilde \rho_K^{n+1/2}) \bigl(\rho_K^{n+1}-\rho_K^n\bigr)\ \sum_{\edge\in \edges(K)} |\edge|\, \rho_\edge^n u_{K,\edge}^n,
\end{align}
where $\tilde \rho_K^{n+1/2}$ is defined by \eqref{tilderho-exp}.
Thus,
\[
\norm{R}_{L^1}=\sum_{n=0}^{N-1} \delta t \sum_{K\in\mesh} |K|\, R_K^{n+1} \leq
|\varphi''|_\infty\, M^2\ \sum_{n=0}^{N-1} \delta t \Bigl(\sum_{K\in\mesh} |\edge|\Bigr)\ |\rho_K^{n+1}-\rho_K^n|,
\]
which yields \eqref{est:R}.

\medskip
Let us now turn to the case where the discrete normal velocities satisfy \eqref{hyp-velocity} and the discretization of the density at the face $ \rho_\edge^n$ is upwind; in this case the remainder $R_2$ defined by \eqref{R2exp} satisfies:
\begin{equation}\label{eq:R_2_upwind}
|K|\,(R_2)_K^{n+1} =  \sum_{\edge=K|L} \frac 1 2\ |\edge|\ \varphi''(\rho_{K,\edge}^n)\ \bigl(\rho_K^n - \rho_L^n \bigr)^2 (u_{K,\edge}^n)^-,
\end{equation}
where $\rho_{K,\edge}^n$ is defined by \eqref{rhoKedge-exp}.
Therefore, $R_2$ is non-negative.
Starting from Equation \eqref{eq:def_R}, we may now reformulate the remainder term $R_K^{n+1}$ as $R_K^{n+1}=(R_{01})_K^{n+1} + (R_{02})_K^{n+1}$ with:
\begin{equation}
\begin{array}{l}\displaystyle
|K|\,(R_{01})_K^{n+1}= \varphi''(\tilde \rho_K^{n+1/2}) \bigl(\rho_K^{n+1}-\rho_K^n\bigr)\ \rho_K^n \bigl[\sum_{\edge\in \edges(K)} |\edge|\ u_{K,\edge}^n\bigr],
\\[3ex] \displaystyle
|K|\,(R_{02})_K^{n+1}=\varphi''(\tilde \rho_K^{n+1/2}) \bigl(\rho_K^{n+1}-\rho_K^n\bigr)
\ \bigl[\sum_{\edge\in \edges(K)} |\edge| (\rho_\edge^n-\rho_K^n) u_{K,\edge}^n\bigr].
\end{array}
\end{equation}
By Young's inequality, the second term may be estimated as follows:
\begin{multline*}
|K|\, |(R_{02})_K^{n+1}| \leq  \frac 1 2 \sum_{\edge\in \edges(K)}
|\edge|\ \varphi''(\rho_{K,\edge}^n)\ (u_{K,\edge}^n)^-\ \bigl(\rho_K^n-\rho_L^n\bigr)^2
\\+ \frac 1 2 \sum_{\edge\in \edges(K)}
|\edge|\ \frac{\varphi''(\tilde \rho_K^{n+1/2})^2}{\varphi''(\rho_{K,\edge}^n)}\ (u_{K,\edge}^n)^- \bigl(\rho_K^{n+1}-\rho_K^n\bigr)^2.
\end{multline*}
Therefore, in view of the expressions \eqref{eq:R_1} and \eqref{eq:R_2_upwind} of $(R_1)_K^n$ and $(R_2)_K^n$ respectively, we get $(R_1)_K^{n+1}+(R_2)_K^{n+1}+ (R_{02})_K^{n+1} \geq 0$ under the CFL condition \eqref{eq:cfl_rho} (with $\varphi$ instead of $\varphi_\rho$).
Let us now show that 
\begin{equation} \label{R01exp}
\norm{R_{01}}_{L^1} \leq f_\mesh\ C_\mesh\ M^{(2q-1)/q}\ |\varphi''|_\infty\ \normtbv{\rho}^{1/q}\ \norm{\bfu}_{L^{q'}(0,T;W^{1,q'}_\mesh)}\ \delta t^{1/q},
\end{equation}
with $q \geq 1$, $q'\geq 1$ and $\dfrac 1 q + \dfrac 1 {q'}=1$, where $f_\mesh$ and $C_\mesh$ is defined by \eqref{eq:def_Cm}.
To this purpose, we first observe that, in the Crouzeix-Raviart or Ranncher-Turek case, since $\sum_{\edge\in \edges(K)} |\edge|\ \bfn_{K,\edge}=0$, we may write:
\[
\sum_{\edge \in \edges(K)} |\edge|\ u^n_{K,\edge}=
\sum_{\edge \in \edges(K)} |\edge|\ \bfu^n_\edge \cdot \bfn_{K,\edge} = \sum_{\edge \in \edges(K)}|\edge|\ (\bfu^n_\edge-\bfu^n_K) \cdot \bfn_{K,\edge},
\]
where $\bfu_K^n$ stands for the mean value of the normal face velocities $(\bfu_\edge^n)_{\edge \in \edges(K)}$.
In the MAC case, we have $\sum_{\edge \in \edges^{(i)}(K)}|\edge|\ u^n_{K,\edge}=0$ for $1 \leq i \leq d$, and thus
\begin{multline*}
\sum_{\edge \in \edges(K)}|\edge|\ u^n_{K,\edge}=
\sum_{i=1}^d \sum_{\edge \in \edges^{(i)}(K)} |\edge|\ u^n_{\edge,i}\ \bfj^{(i)} \cdot \bfn_{K,\edge}=
\\
\sum_{i=1}^d \sum_{\edge \in \edges^{(i)}(K)} |\edge|\ (u^n_{\edge,i}-u^n_{K,i})\ \bfj^{(i)} \cdot \bfn_{K,\edge},
\end{multline*}
where $\bfj^{(i)}$ stands for the $i^{th}$ vector of the canonical basis of $\xR^d$ and $u^n_{K,i}$ stands for the mean value of the velocity over the two faces of $\edges^{(i)}(K)$.
In both cases, we obtain that:
\[
\Bigl| \sum_{\edge\in \edges(K)} |\edge|\ u^n_{K,\edge} \Bigr| \leq
2 \sum_{i=1}^d\ \sum_{(\edge,\edge') \in \edges^{(i)}(K)^2} (|\edge|+|\edge'|)\ |u^n_{\edge,i}-u^n_{\edge',i}|.
\]
Therefore, 
\begin{multline*}
|K|\,|(R_{01})_K^{n+1}| \leq 2\ |\varphi''|_\infty M\ \bigl|\rho_K^{n+1}-\rho_K^n\bigr|
\\
\sum_{i=1}^d\ \sum_{(\edge,\edge') \in \edges^{(i)}(K)^2} (|\edge|+|\edge'|)\ |u^n_{\edge,i}-u^n_{\edge',i}|.
\end{multline*}
We thus have, thanks to a H\"older estimate, for $q \geq 1$, $q' \geq 1$ and $\dfrac 1 q+ \dfrac 1 {q'}=1$:
\begin{multline*}
\norm{R_{01}}_{L^1}  = \sum_{n=0}^{N-1} \delta t \sum_{K\in\mesh} |K|\,(R_{01})_K^{n+1}
\\
\leq  2\ |\varphi''|_\infty M\ \Bigl[ \delta t \sum_{n=0}^{N-1} \sum_{K\in\mesh} |K|\ \bigl|\rho_K^{n+1}-\rho_K^n\bigr|^q
\Bigl(\sum_{i=1}^d\ \sum_{(\edge,\edge') \in \edges^{(i)}(K)^2} 1\Bigr) \Bigr]^{1/q}
\\
\Bigl[\sum_{n=0}^{N-1} \sum_{K\in\mesh} \sum_{i=1}^d \sum_{(\edge,\edge') \in \edges^{(i)}(K)^2}  
\delta t\ |K| \ \left(\frac{|u^n_{\edge,i}-u^n_{\edge',i}|}{h_K}\right)^{q'} \left(\frac{(|\edge|+|\edge'|)\,h_K}{|K|}\right)^{q'}
\Bigr]^{1/q'}.
\end{multline*}
Using $|\rho_K^{n+1}-\rho_K^n |^q \leq (2 M)^{q-1}\ |\rho_K^{n+1}-\rho_K^n |$ yields \eqref{R01exp}. 
\end{proof}
%
%
The object of the following proposition is to mimick at the fully discrete level the computation of \eqref{eq:tr} and \eqref{eq:cvz}, applying it to $z= e$. 

\begin{proposition}[Inequalities derived from the internal energy balance]\label{prop:e}
Let $M \geq 1$ and let us suppose that $\rho_K^n \leq M$, $e_K^n<M$, $1/ e_K^n \leq M$ and $|u_{K,\edge}| \leq M$, for $K \in \mesh$, $\edge \in \edges(K)$ and $0 \leq n \leq N$.
Assume that the face approximation of the internal energy satisfies \eqref{eq:H_rho^exp}.
Let $\varphi$ be a twice continuously convex function from $(0,+\infty)$ to $\xR$.
Then 
\begin{multline}\label{eq:e_id_exp} 
\varphi'(e_K^{n+1})\ \Bigl[\frac{|K|}{\delta t} (\rho_K^{n+1}e_K^{n+1}-\rho_K^n e_K^n) + \sum_{\edge\in \edges(K)} F_{K,\edge}^n e_\edge^n \Bigr]
\\
\geq \frac{|K|}{\delta t} \bigl(\rho_K^{n+1}\varphi(e_K^{n+1})-\rho_K^n\, \varphi(e_K^n) \bigr)
+ \sum_{\edge\in \edges(K)} F_{K,\edge}^n\, \varphi(e_\edge^n)
+ |K|\,(R_e)_K^n,
\hspace{3ex}\end{multline}
with 
\begin{equation}
\norm{R_e}_{L^1} \leq 3M^2\ |\varphi'|_\infty\, \normxbv{e}\ h_\mesh + M^2\ |\varphi''|_\infty\,\normtbv{e}\ \frac{\delta t}{\underline h_\mesh}, \label{eq:remainder-exp-e}
\end{equation}
where $|\varphi'|_\infty=\max(|\varphi'(1/M)|,\ |\varphi'(M)|)$,  $|\varphi''|_\infty$ stands for the maximum value taken by $\varphi''$ over the interval $[1/M,\ M]$, and $\underline h_\mesh$ is defined by \eqref{eq:def_h_bar}.
If, furthermore, the approximation of $e_\edge^n$ in \eqref{eq:e_id_exp} is upwind, and the CFL condition \eqref{eq:cfl_e} holds (with $\varphi$ instead of $\varphi_e$) then $(R_e)_K^n=0$.
\end{proposition}

\begin{proof}
First, the fully discrete identity corresponding to the semi-discrete identity \eqref{ident-convec} with $z=e$ is obtained thanks to the discrete mass equation; it reads:
\begin{multline*}
\frac{|K|}{\delta t} (\rho_K^{n+1} e_K^{n+1}-\rho_K^n e_K^n) + \sum_{\edge\in \edges(K)} F_{K,\edge}^n e_\edge^n  =
\frac{|K|}{\delta t} \rho_K^{n+1} (e_K^{n+1}-e_K^n)
\\
+ \sum_{\edge\in \edges(K)} F_{K,\edge}^n (e_\edge^n-e_K^n), \qquad \forall K \in \mesh, \quad 0 \leq n \leq N-1.
\end{multline*}
Now let  $\varphi$ be a twice continuously differentiable function from$(0,+\infty)$ to $\xR$, and let us multiply the first two terms of the discrete internal energy balance \eqref{eq:e_int_e} by $\varphi'(e_K^{n+1})$; switching from the conservative to the non conservative form, we get:
\begin{multline*}
\varphi'(e_K^{n+1})\ \Bigl[\frac{|K|}{\delta t} (\rho_K^{n+1}e_K^{n+1}-\rho_K^n e_K^n) + \sum_{\edge\in \edges(K)} F_{K,\edge}^n (e_\edge^n -e_K^n)\Bigr]
\\
=(T_1)_K^{n+1}+(T_2)_K^{n+1}+ |K|\,R_K^{n+1},
\end{multline*}
with
\begin{align} \nonumber &
(T_1)_K^{n+1}=\varphi'(e_K^{n+1})\ \frac{|K|}{\delta t} \rho_K^{n+1}(e_K^{n+1}-e_K^n),
\displaybreak[1] \\ \nonumber & 
(T_2)_K^{n+1}= \varphi'(e_K^n) \sum_{\edge\in \edges(K)} F_{K,\edge}^n (e_\edge^n - e_K^n),
\displaybreak[1] \\ \label{eq:def_eren} &
|K|\,R_K^{n+1}=\bigl(\varphi'(e_K^{n+1})-\varphi'(e_K^n)\bigr)\ \sum_{\edge\in \edges(K)} F_{K,\edge}^n(e_\edge^n -e_K^n).
\end{align}
The remainder term $R_K^{n+1}$ is quite similar to the remainder defined by \eqref{eq:def_ren} in the proof of Proposition \ref{prop:rho}; following the proof of \eqref{est:R}, we get that it satisfies:
\[
\norm{R}_{L^1} \leq M^2\ |\varphi''|_\infty\ \normtbv{e}\ \frac{\delta t}{\underline h_\mesh}.
\]
Now
\begin{align*}
&(T_1)_K^{n+1}= \frac{|K|}{\delta t} \rho_K^{n+1} \bigl(\varphi(e_K^{n+1})-\varphi(e_K^n) \bigr) +  |K|\,(R_1)_K^{n+1}, 
\\	
&(T_2)_K^{n+1}= \sum_{\edge\in \edges(K)} F_{K,\edge}^n \bigl( \varphi(e_\edge^n) - \varphi(e_K^n)\bigr)+ |K|\,(R_2)_K^{n+1},
\end{align*}
with:
\[\begin{array}{l}\displaystyle
|K|\,(R_1)_K^{n+1}= \frac{|K|}{\delta t} \rho_K^{n+1}\, \bigl(\varphi(e_K^n) - \varphi(e_K^{n+1}) - \varphi'(e_K^{n+1})( e_K^n-e_K^{n+1})\bigr),
\\[3ex] \displaystyle
|K|\,(R_2)_K^{n+1}= \sum_{\edge\in \edges(K)} F_{K,\edge}^n\, \bigl(\varphi(e_K^n) + \varphi'(e_K^n) (e_\edge^n - e_K^n)- \varphi(e_\edge^n)\bigr).
\end{array}\]
The remainder $(R_1)_K^{n+1}$ may be written:
\begin{equation}\label{eq:R1-e-eup}
(R_1)_K^{n+1} = \frac 1 {2\delta t} \rho_K^{n+1}\, \varphi''(e_K^{n+1/2})\ (e_K^{n+1}-e_K^n)^2,
\end{equation}
where $e_K^{n+1/2} \in \li e_K^n,\ e_K^{n+1} \ri$.
Since $\varphi$ is supposed to be convex, this term is non-negative.
Let $e_{KL}^n$ be the real number defined by Equation \eqref{eq:conv_int} (and denoted in this latter relation by $x_{KL}$) with $x_K=e^n_K$ and $x_L=e^n_L$.
Thanks to \eqref{eq:H_rho^exp}, by a computation similar to the implicit case, the remainder $R_2$ satisfies
\begin{equation}\label{oldlem:R2-e}
R_2 \geq \delta\! R_2 \mbox{ and} \quad \norm{\delta\! R_2}_{L^1} \leq 3 M^2\ |\varphi'|_\infty\, \normxbv{e}\ h_\mesh.
\end{equation}
Switching back from the non-conservative formulation to the conservative formulation yields:
\begin{multline}
\varphi'(e_K^{n+1})\ \frac{|K|}{\delta t} \rho_K^{n+1} (e_K^{n+1} - e_K^n)
+ \varphi'(e_K^n) \Bigl[ \sum_{\edge\in \edges(K)} F_{K,\edge}^n (e_\edge^n - e_K^n) \Bigr] \ge 
\\
\frac{|K|}{\delta t} \bigl(\rho_K^{n+1}\varphi(e_K^{n+1})-\rho_K^n\, \varphi(e_K^n) \bigr)
+ \sum_{\edge\in \edges(K)} F_{K,\edge}^n\, \varphi(e_\edge^n)
+ |K|(R_2)_K^{n+1},
\end{multline}
which, thanks to \eqref{oldlem:R2-e}, leads to \eqref{eq:e_id_exp} and \eqref{eq:remainder-exp-e}.

\medskip
Let us now suppose that the discretization of the internal energy convection term is upwind.
In this case, we obtain for $(R_2)_K^{n+1}$:
\begin{equation}\label{eq:R2-e-eup}
|K|\,(R_2)_K^{n+1}=\frac 1 2 \sum_{\edge \in \edges(K)} (F_{K,\edge}^n)^-\, \varphi''(e_{K,\edge}^n) (e_K^n-e_L^n)^2,
\end{equation}
where $e_{K,\edge}^n \in \li e_K^n,\ e_L^n \ri$.
The remainder $R_K^{n+1}$ yields in the upwind case:
\[
|K|\, R_K^{n+1} = -\varphi''(\tilde e_K^{n+1/2})\ (e_K^{n+1}-e_K^n)\ \Bigl[ \sum_{\edge\in \edges(K)} (F_{K,\edge}^n)^- (e_L^n - e_K^n) \Bigr],
\]
where $e_K^{n+1/2} \in \li e_K^n,\ e_K^{n+1} \ri$.
So, thanks to the Young inequality:
\begin{multline*}
|K|\,|R_K^{n+1}| \leq \frac 1 2 \sum_{\edge\in \edges(K)} (F_{K,\edge}^n)^- \, \varphi''(e_{K,\edge}^n) (e_L^n - e_K^n)^2
\\ + \frac 1 2\ (e_K^{n+1}-e_K^n)^2 \sum_{\edge\in \edges(K)} (F_{K,\edge}^n)^- \frac{\varphi''(\tilde e_K^{n+1/2})^2}{\varphi''(e_{K,\edge}^n)}.
\end{multline*}
In view of the expressions \eqref{eq:R1-e-eup} and \eqref{eq:R2-e-eup} of $(R_1)_K^{n+1}$ and $(R_2)_K^{n+1}$ respectively, we obtain that $(R_1)_K^{n+1}+(R_2)_K^{n+1}+R_K^{n+1}\geq 0$ thanks to the CFL condition \eqref{eq:cfl_e}, which yields the result.
\end{proof}
 
\medskip
Theorem \ref{theo:entropie} deserves the following comments:
\bli
\item First, in the explicit case, we are able to prove neither a local nor a global discrete entropy inequality;  we only obtain some weak inequalities that allow to show the consistency of the scheme, under some conditions.
\item The convergence to zero with the space and time step of the remainders is obtained, supposing a control of discrete solutions in $L^\infty$ and discrete BV norms, in two cases: first when the ratio $\delta t/\underline h_\mesh$ tends to zero, second when the $L^{q}(0,T;W^{1,q}_\mesh)$ norm of the velocity does not blow-up too quickly with the space step.
To this respect, let us suppose that we implement a stabilization term in the momentum balance equation reading (in a pseudo-continuous setting, for short and to avoid the technicalities associated to the space discretization), for $1 \leq i \leq d$:
\begin{equation}\label{eq:mom}
\partial_t(\rho u_i) + \dive(\rho u_i \bfu) + \partial_i p - h_\mesh^\alpha \Delta_{q} u_i =0,
\end{equation}
where $\Delta_q u_i$ is such that
\[
\norm{u_i}_{W^{1,q}_\mesh}^{q} \leq C \int_\Omega -\Delta_{q} u_i\ u_i \dx,
\]
where $C$ is independent of $h_\mesh$.
This kind of viscosity term may be found in turbulence models \cite{berselli2006math,sagaut2006large}.
Multiplying \eqref{eq:mom} by $u_i$ and integrating with respect to space and time yields:
\begin{equation}\label{eq:mom_int}
\int_0^T \int_\Omega -\Delta_{q} u_i\ u_i \dx \dt =
- \int_0^T \int_\Omega \bigl(\partial_t(\rho u_i) + \dive(\rho u_i \bfu) + \partial_i p\bigr)\ u_i \dx \dt.
\end{equation}
In this relation, the right-hand side may be controlled under $L^\infty$ and BV stability assumptions (remember that, at the discrete level, the BV and $W^{1,1}$ norms are the same), and we obtain an estimate on $\norm{\bfu}_{L^{q}(0,T;W^{1,q}_\mesh)}$ which may be used in \eqref{eq:R_ent_e}.
A standard first order diffusion-like stabilizing term corresponds to $q=2$ and $\alpha=1$; it yields a bound on $h_\mesh^{1/2} \norm{\bfu}_{L^2(0,T;H^1_\mesh)}$, so that \eqref{eq:R_ent_e} becomes
\[
\norm{R_\eta}_{L^1} \leq f_\mesh\ C_\mesh\ M^{\frac{3}{2}}\ |\varphi''|_\infty\ \normtbv{\rho}^{\frac 1 2}\ \tilde C(\dfrac {\delta t}{h_\mesh})^{\frac 1 2}.
\] 
Such a stabilization is thus not sufficient to ensure that the remainder term tends to zero.
What is needed is in fact:
\[
\alpha < q-1.
\]
To avoid an over-diffusion in the momentum balance, this inequality suggests to implement a non-linear stabilization with $q>2$ which, in turn, will allow $\alpha >1$.
With such a trick, we will be able to obtain for first-order upwind schemes the desired "Lax-consistency" result: the limit of a convergent sequence of solutions, bounded in $L^\infty$ and BV norms, and obtained with space and time steps tending to zero, satisfies a weak entropy inequality.
\item We introduced in \cite{pia-13-for} a limitation process for a MUSCL-like algorithm for the transport equation, which consists in deriving an admissible interval for the approximation of the unknowns at the mesh faces, in convection terms, thanks to extrema preservation arguments.
This limitation process has been extended to the Euler equations in \cite{gas-18-mus}.
The conditions \eqref{eq:H_rho^exp} and \eqref{eq:H_e^exp} may easily be incorporated in this limitation: indeed, they also define an admissible interval, which is not disjoint from the MUSCL-like admissible interval of \cite{pia-13-for}, since the upwind value belongs to both.
A similar idea (namely restricting the choice for the face approximation in order to obtain an entropy inequality) may be found in \cite{ber-14-ent}.
\end{list}
--------------------------------------------------------------------------------------------------------
\bibliographystyle{splncs04}
\bibliography{./grancanaria}
\end{document}